\documentclass[notitlepage,11pt,reqno]{amsart}
\usepackage[foot]{amsaddr}
\usepackage{amssymb,nicefrac,upgreek,mathtools,verbatim}
\usepackage[final]{hyperref}
\usepackage[mathscr]{eucal}
\usepackage{dsfont}
\usepackage[normalem]{ulem}
\usepackage{amsopn}
 
\usepackage[margin=1in]{geometry}
\allowdisplaybreaks
\newcommand{\stkout}[1]{\ifmmode\text{\sout{\ensuremath{#1}}}\else\sout{#1}\fi}

\newtheorem{lemma}{Lemma}[section]
\newtheorem{theorem}{Theorem}[section]

\theoremstyle{definition}
\newtheorem{definition}{Definition}[section]

\theoremstyle{remark}
\newtheorem{remark}{Remark}[section]
\numberwithin{theorem}{section}
\numberwithin{equation}{section}

\hypersetup{
  colorlinks=true,
  citecolor=mblue,
  linkcolor=mblue,
  urlcolor = blue,
  anchorcolor = blue,
  frenchlinks=false,
  pdfborder={0 0 0},
  naturalnames=false,
  hypertexnames=false,
  breaklinks}
\usepackage[capitalize,nameinlink]{cleveref}
\usepackage[abbrev,msc-links,nobysame,citation-order]{amsrefs} 

\crefname{section}{Section}{Sections}
\crefname{subsection}{Section}{Sections}
\crefname{condition}{Condition}{Conditions}
\crefname{hypothesis}{Hypothesis}{Conditions}
\crefname{assumption}{Assumption}{Assumptions}
\crefname{lemma}{Lemma}{Lemmas}
\crefname{fact}{Fact}{Facts}

\Crefname{figure}{Figure}{Figures}

\crefformat{equation}{\textup{#2(#1)#3}}
\crefrangeformat{equation}{\textup{#3(#1)#4--#5(#2)#6}}
\crefmultiformat{equation}{\textup{#2(#1)#3}}{ and \textup{#2(#1)#3}}
{, \textup{#2(#1)#3}}{, and \textup{#2(#1)#3}}
\crefrangemultiformat{equation}{\textup{#3(#1)#4--#5(#2)#6}}%
{ and \textup{#3(#1)#4--#5(#2)#6}}{, \textup{#3(#1)#4--#5(#2)#6}}%
{, and \textup{#3(#1)#4--#5(#2)#6}}

\Crefformat{equation}{#2Equation~\textup{(#1)}#3}
\Crefrangeformat{equation}{Equations~\textup{#3(#1)#4--#5(#2)#6}}
\Crefmultiformat{equation}{Equations~\textup{#2(#1)#3}}{ and \textup{#2(#1)#3}}
{, \textup{#2(#1)#3}}{, and \textup{#2(#1)#3}}
\Crefrangemultiformat{equation}{Equations~\textup{#3(#1)#4--#5(#2)#6}}%
{ and \textup{#3(#1)#4--#5(#2)#6}}{, \textup{#3(#1)#4--#5(#2)#6}}%
{, and \textup{#3(#1)#4--#5(#2)#6}}

\crefdefaultlabelformat{#2\textup{#1}#3}
\newcommand{\vertiii}[1]{{\left\vert\kern-0.25ex\left\vert\kern-0.25ex\left\vert #1 
    \right\vert\kern-0.25ex\right\vert\kern-0.25ex\right\vert}}

\newcommand{\Uadm}{\mathfrak U}
\newcommand{\Act}{\mathbb{U}}
\newcommand{\Usm}{\mathfrak U_{\mathsf{sm}}}

\newcommand{\Um}{\mathfrak U_{\mathsf{m}}}

\newcommand{\pV}{\mathrm{V}} 
\newcommand{\pv}{\mathrm{v}} 

\newcommand{\sB}{{\mathscr{B}}}  
\newcommand{\cC}{{\mathcal{C}}}   
\newcommand{\sE}{{\mathscr{E}}} 
\newcommand{\sF}{{\mathfrak{F}}}   
\newcommand{\cJ}{{\mathcal{J}}}  
  
\newcommand{\sL}{{\mathscr{L}}}  %
\newcommand{\Lp}{{L}}            

\newcommand{\Lyap}{{\mathcal{V}}}  

\newcommand{\RR}{\mathds{R}}
\newcommand{\NN}{\mathds{N}}

\newcommand{\Rd}{{\mathds{R}^{d}}}
\DeclareMathOperator{\Exp}{\mathbb{E}}

\newcommand{\D}{\mathrm{d}}

\newcommand{\Ind}{\mathds{1}}   
\newcommand{\cD}{\mathcal{D}} 
 
\newcommand{\Sob}{{\mathscr W}}    
\newcommand{\Sobl}{{\mathscr W}_{\text{loc}}} 

\newcommand{\df}{:=}
\newcommand{\transp}{^{\mathsf{T}}}
\DeclareMathOperator*{\osc}{osc}

\DeclareMathOperator*{\trace}{Tr}

\DeclareMathOperator*{\diam}{diam}

\DeclareMathOperator*{\argmin}{arg\,min}

\newcommand{\sorder}{{\mathfrak{o}}}
\newcommand{\grad}{\nabla}
\newcommand{\uuptau}{{\Breve\uptau}}

\newcommand{\abs}[1]{\lvert#1\rvert}
\newcommand{\norm}[1]{\lVert#1\rVert}

\usepackage{color}
\definecolor{dmagenta}{rgb}{.4,.1,.5}
\definecolor{dblue}{rgb}{.0,.0,.5}
\definecolor{mblue}{rgb}{.0,.0,.7}
\definecolor{ddblue}{rgb}{.0,.0,.4}
\definecolor{dred}{rgb}{.7,.0,.0}
\definecolor{dgreen}{rgb}{.0,.5,.0}
\definecolor{Eeom}{rgb}{.0,.0,.5}

\begin{document}
\title[Near optimality of quantized policies]
{Continuity of Cost in Borkar Control Topology and Implications on Discrete Space and Time Approximations for Controlled Diffusions under Several Criteria}

\author[Somnath Pradhan]{Somnath Pradhan$^\dag$}
\address{$^\dag$Department of Mathematics and Statistics,
Queen's University, Kingston, ON, Canada}
\email{sp165@queensu.ca}

\author[Serdar Y\"{u}ksel]{Serdar Y\"{u}ksel$^{\ddag}$}
\address{$^\ddag$Department of Mathematics and Statistics,
Queen's University, Kingston, ON, Canada}
\email{yuksel@queensu.ca}

\begin{abstract}
We first show that the discounted cost, cost up to an exit time, and ergodic cost involving controlled non-degenerate diffusions are continuous on the space of stationary control policies when the policies are given a topology introduced by Borkar [V. S. Borkar, A topology for markov controls, Applied Mathematics and Optimization 20 (1989), 55–62]. The same applies for finite horizon problems when the control policies are Markov and the topology is revised to include time also as a parameter. We then establish that finite action/piecewise constant stationary policies are dense in the space of stationary Markov policies under this topology and the same holds for continuous policies. Using the above mentioned continuity and denseness results we establish that finite action/piecewise constant policies approximate optimal stationary policies with arbitrary precision. This gives rise to the applicability of many numerical methods such as policy iteration and stochastic learning methods for discounted cost, cost up to an exit time, and ergodic cost optimal control problems in continuous-time. For the finite-horizon setup, we establish additionally near optimality of time-discretized policies by an analogous argument. We thus present a unified and concise approach for approximations directly applicable under several commonly adopted cost criteria.
\end{abstract}


\maketitle

\section{Introduction}

In this paper, we study regularity properties of induced cost (under several criteria) on a controlled diffusion process with respect to a control topology defined by Borkar \cite{Bor89}, and implications of these properties on existence and, in particular, approximations for optimal controlled diffusions. We will arrive at very general approximation results for optimal control policies by quantized (finite action / piecewise constant) stationary control policies for a general class of controlled diffusions in the whole space $\Rd$\, as well as time-discretizations for the criteria with finite horizons. 

Such a problem is of significant practical consequence, and accordingly has been studied extensively in a variety of setups. Due to its wide range of applications in domains that spans from  mathematical ﬁnance, large deviations and robust control, vehicle and mobile robot control and several other ﬁelds, the stochastic optimal control problems for controlled diffusions have been studied extensively in literature see, e.g., \cite{Bor-book}, \cite{HP09-book} (finite horizon cost) \cite{BS86}, \cite{BB96} (discounted cost) \cite{Ari-12}, \cite{AA13}, \cite{BG90}, \cite{BG88I}, \cite{BG90b}, \cite{ABG-book} (ergodic cost) and references therein\,. Typically, there are two main approaches to deal with these problems. The first one is the Bellman's Dynamic Programming Principal (DPP). The DPP approach allows one to characterize the value function of the optimal control problem as the unique solution of the associated Hamilton-Jacobi-Bellman (HJB) equation \cite{Bor-book}, \cite{HP09-book}, \cite{ABG-book}, \cite{Lions-83A}, \cite{Lions-83B}. The second one is Pontryagin maximum principal (in the stochastic framework) \cite{PBGM62}\,.     

For numerical methods as well as learning theoretic methods, it is imperative to arrive at rigorous approximation results.

In the continuous-time literature, most of the approximation results are build on time-discretization and mainly focused on finite horizon or discounted cost criteria see, e.g., \cite{KD92}, \cite{KH77}, \cite{KH01}, \cite{KN98A}, \cite{KN2000A}, \cite{BR-02}, \cite{BJ-06}\,, though the ergodic control and control up to an exit time criteria have also been studied \cites{KD92,kushner2014partial}. 

For finite horizon criteria, a commonly adopted approach of approximating controlled diffusions by a sequence of discrete time Markov chain via weak convergence methods was studied by Kushner and Kushner and Dupuis, see \cite{KD92}, \cite{KH77}, \cite{KH01}\,. These works deal with numerical procedures to construct near optimal  control policies for controlled diffusion models by approximating the space of (open-loop adapted) relaxed control policies with those that are piece-wise constant, and by considering the weak convergence of approximating probability measures on the path space to the measure on the continuous-time limit. It is shown in \cite{KD92}, \cite{KH77}, \cite{KH01} that if the constructed controlled Markov chain satisfies a certain ``consistency" condition at the discrete-time sampling instants, then the state process and the corresponding value function asymptotically approximates the continuous time state process and the associated value function. This approach has been referred to as the {\it weak convergence} approach. 

In an alternative program, building on finite difference approximations for Bellman's equations utilizing their regularity properties, Krylov \cite{KN98A}, \cite{KN2000A} established the convergence rate of for such approximation techniques, where finite difference approximations are studied to arrive at stability results. In particular, some estimates for the error bound of the finite-difference approximation schemes in the problem of finding viscosity or probabilistic solutions to degenerate Bellman's equations are established. The proof technique is based on mean value theorems for stochastic integrals (as in \cite{KN2001AA}), obtained on the basis of elementary properties of the associated Bellman's equations\,. Also, for controlled non-degenerate diffusion processes, it is shown in \cite{KN99AA} that using policies which are constant on intervals of length $h^2$, one can approximate the value function with errors of order $h^{\frac{1}{3}}$\,. In \cite{BR-02}, \cite{BJ-06} Barles et. al. improved the error bounds obtained in \cite{KN98A}, \cite{KN2000A}, \cite{KN99AA}\,. 

Borkar \cite{Bor89}, \cite{Bor-book}, for the finite-horizon cost case pursued an alternative approach to show continuity (when only stationary state feedback policies are considered for finite horizon problems) in his newly introduced topology; he studied the dependence of the strategic measures (on the path space) on the control policy, via regularity properties of generator functions. Additionally, Borkar \cite{Bor89} did not study the implications in approximations.  

Instead of the approaches adopted in the aforementioned studies, in this paper, utilizing regularity results of the associated Poisson equations via PDE theory, we arrive at continuity results under relatively weaker set of assumptions on the diffusion coefficients (with the exception of Krylov's method, which is tailored for finite horizon problems). Our approach allows one to arrive at a unification of approximation methods for finite horizon criterion, infinite discounted criterion, control up to an exit time, and ergodic cost criterion problems. Accordingly, our primary approach is to utilize the regularity properties of the partial differential equations directly, first via uniqueness of solutions, and then via regularity properties of the solutions to establish consistency of optimality equations satisfied by the limits of solutions (as policies converge). We will see that one can obtain rather concise, direct, and general results. 

Additionally, our results can be used to present weaker conditions under which the weak convergence methods can be applicable or when discretized approximations can be shown to be near optimal: For example it will be a consequence of our analysis that for many of the criteria one can utilize piece-wise continuous or continuous control policies for near optimality, which implies \cite[Assumption A2.3, pp. 322]{KD92} used for approximations under ergodic cost criteria (where invariant measures under sampled chains can be shown to converge to the invariant measure of a continuous-time limit as discretization gets finer). Furthermore, we do not impose uniform boundedness conditions on the drift term or (uniform) Lipschitz continuity conditions, a common assumption in \cite{KD92}, \cite{KH77}, \cite{KH01}, \cite{KN98A}, and \cite{KN2000A}.

As noted above, the study of the finite action/piecewise constant approximation problem plays important role in computing near optimal policies and learning algorithms for controlled diffusions in $\Rd$\,. As it is pointed out in \cite{RF-15N}, \cite{JPR-19P}, piecewise constant policies are also useful in numerical methods for solving HJB equations. The computational advantage comes from the fact that over the intervals in which the policy is constant, we have to only solve the linear PDEs\,.  In the continuous time setup learning problems become much more involved due to the complex structure of the dynamics and the optimality equation. One common approach to overcome these difficulties is to construct simpler models by discretizing time, space and action spaces which approximates the original continuous time model\,. In a recent work \cite{BK-22R}, the authors studied an approximate $Q$-learning algorithm for controlled diffusion models by discretizing the time, space and action spaces. Under mild assumptions, they produced a learning algorithm which converges to some approximately optimal control policy for a discounted cost problem. They assumed that the discretization is uniform in time but the discretization in state and action can be non-uniform\,. Similar learning algorithm for controlled diffusions is proposed in \cite{RP-97B}, this result is based on the finite difference and finite element approximations (as in,\cite{KD92})\,. Thus, if one can establish that learning a control model with finitely many control actions is sufficient for the approximate optimality, then it will be easier to produce efficient learning algorithms for the original model\,.

In the literature of discrete time Markov decision processes (MDPs), various approximation techniques are available to address the approximation problems, e.g., approximate dynamic programming, approximate value or policy iteration, approximate linear programming, simulation based techniques, neuro-dynamic programming (or reinforcement learning), state aggregation, etc. (see \cite{Bertsekas1975A}, \cite{BVR-06}, \cite{DP-13SIAM}, \cite{SLS-18B} and the references therein)\,. For discrete time controlled models the near optimality of quantized policies studied extensively in the literature see, e.g., \cite{KY22A}, \cite{SYL17}, \cite{SSL-16}, \cite{SST-20A}, \cite{SLS-18B}   \,. In \cite{SYL17}, \cite{SSL-16}, authors studied the finite state, finite action approximation (respectively) of fully observed MDPs with Borel state and action spaces, for both discounted and average costs criteria\,. In the compact state space case explicit rate of convergence is also established in \cite{SYL17}\,. Later, these results are extended to partially observed Markov decision process setup in \cite{KY22A}, \cite{SST-20A}, also see the references therein\,. Recently, \cite[Section 4]{arapostathis2021optimality} established the denseness of the performance of deterministic policies with finite action spaces, among the performance values attained by the set of all randomized stationary policies.

\subsection*{Contributions and main results} In this manuscript our main goal is to study the following approximation problem: for a general class of controlled diffusions in $\Rd$ under what conditions one can approximate the optimal control policies for both finite/infinite horizon cost criteria by policies with finite actions/ piecewise constant/continuous policies? While the time discretization approximation results for finite horizon problems, studied extensively by Krylov \cite{KN98A}, \cite{KN2000A}, \cite{KN99AA} (for degenerate diffusions), we will discuss this (for the non-degenerate case) as an application our results. 

In order to address these questions, we first show that both finite horizon and infinite horizon (discounted/ergodic) costs are continuous as a function of control policies under Borkar topology \cite{Bor89}. We establish these results by exploiting the existence and uniqueness results of the associated Poisson equations (see,  Theorem~\ref{TContFHC} (finite horizon), Theorem~\ref{T1.1} (discounted), Theorem~\ref{T1.1Exit} (control up to an exit time), Theorem~\ref{ergodicnearmono1}, \ref{ergodicLyap1} (ergodic)). The  analysis of ergodic cost case is relatively more involved. One of the major issues in analyzing the ergodic cost criteria under the near-monotone hypothesis is the non-uniqueness/restricted uniqueness of the solution of the associated HJB/Poisson equation (see, \cite[Example~3.8.3]{ABG-book},\cite{AA13})\,. In \cite[Example~3.8.3]{ABG-book},\cite{AA13} it is shown that under near-monotone hypothesis the associated HJB/Poisson equation may admit uncountable many solutions\,. In this paper, we have shown that under near-monotone hypothesis the associated Poisson equation admits unique solution in the space of compatible solution pairs (see, \cite[Definition~1.1]{AA13})\,. Continuity results obtained in the paper will be also useful in establishing the existence of optimal policies of the corresponding optimal control problems. 

Next, utilizing the Lusin's theorem and Tietze's extension theorem we show that under Borkar topology, quantized (finite actions/ piecewise constant) stationary policies are dense in the space of stationary Markov policies (see, Section~\ref{DensePol})\,. Also, following the analogous proof technique, we establish the denseness of space continuous stationary polices in the space of stationary policies (see Theorem~\ref{TDContP})\,. 

Following and briefly modifying the proof technique of the denseness of stationary policies, including time also as a parameter we establish that piecewise constant Markov policies are dense in the space of Markov policies under Borkar topology (see, Theorem~\ref{TDPCMP}). 

Then, using our continuity and denseness results, we deduce that for both finite and infinite horizon cost criteria, the optimal control policies can be approximated by quantized (finite actions/ piecewise constant) policies with arbitrary precision (see, Theorem~\ref{TFiniteOptApprox1} (finite horizon), Theorem~\ref{T1.2ExitCost} (control upto an exit time), Theorem~\ref{ErgodNearmOPT1}, \ref{TErgoOptApprox1} (infinite horizon)).


The remaining part of the paper is organized as follows. In Section~\ref{PD} we provide the problem formulation\,. The continuity of discounted cost/ cost up to an exit time as a function of control policy are proved in Section~\ref{CDiscCost}. Similar continuity result for ergodic cost is presented in Section~\ref{CErgoCost}, where we establish these results under two types of condition; stability or near-monotonicity. Section~\ref{DensePol} is devoted to establish the denseness of finite action/piecewise constant stationary policies under Borkar topology. Then using the denseness and continuity results we show the near optimality of finite models for cost up to an exit time and discounted/ ergodic cost criteria in Section~\ref{NOptiFinite}. Finally, in Section~\ref{TimeDMarkov}, we analyze the denseness of piecewise constant Markov policies under Borkar topology and then exploiting the denseness result we prove the near optimality of the piecewise constant Markov policies for finite horizon cost criterion.  
\subsection*{Notation:}
\begin{itemize}
\item For any set $A\subset\RR^{d}$, by $\uptau(A)$ we denote \emph{first exit time} of the process $\{X_{t}\}$ from the set $A\subset\RR^{d}$, defined by
\begin{equation*}
\uptau(A) \,\df\, \inf\,\{t>0\,\colon X_{t}\not\in A\}\,.
\end{equation*}
\item $\sB_{r}$ denotes the open ball of radius $r$ in $\RR^{d}$, centered at the origin,
\item $\uptau_{r}$, $\uuptau_{r}$ denote the first exist time from $\sB_{r}$, $\sB_{r}^c$ respectively, i.e., $\uptau_{r}\df \uptau(\sB_{r})$, and $\uuptau_{r}\df \uptau(\sB^{c}_{r})$.
\item By $\trace S$ we denote the trace of a square matrix $S$.
\item For any domain $\cD\subset\RR^{d}$, the space $\cC^{k}(\cD)$ ($\cC^{\infty}(\cD)$), $k\ge 0$, denotes the class of all real-valued functions on $\cD$ whose partial derivatives up to and including order $k$ (of any order) exist and are continuous.
\item $\cC_{\mathrm{c}}^k(\cD)$ denotes the subset of $\cC^{k}(\cD)$, $0\le k\le \infty$, consisting of functions that have compact support. This denotes the space of test functions.
\item $\cC_{b}(\Rd)$ denotes the class of bounded continuous functions on $\Rd$\,.
\item $\cC^{k}_{0}(\cD)$, denotes the subspace of $\cC^{k}(\cD)$, $0\le k < \infty$, consisting of functions that vanish in $\cD^c$.
\item $\cC^{k,r}(\cD)$, denotes the class of functions whose partial derivatives up to order $k$ are H\"older continuous of order $r$.
\item $\Lp^{p}(\cD)$, $p\in[1,\infty)$, denotes the Banach space
of (equivalence classes of) measurable functions $f$ satisfying
$\int_{\cD} \abs{f(x)}^{p}\,\D{x}<\infty$.
\item $\Sob^{k,p}(\cD)$, $k\ge0$, $p\ge1$ denotes the standard Sobolev space of functions on $\cD$ whose weak derivatives up to order $k$ are in $\Lp^{p}(\cD)$, equipped with its natural norm (see, \cite{Adams})\,.
\item  If $\mathcal{X}(Q)$ is a space of real-valued functions on $Q$, $\mathcal{X}_{\mathrm{loc}}(Q)$ consists of all functions $f$ such that $f\varphi\in\mathcal{X}(Q)$ for every $\varphi\in\cC_{\mathrm{c}}^{\infty}(Q)$. In a similar fashion, we define $\Sobl^{k, p}(\cD)$.
\item  For $\mu > 0$, let $e_{\mu}(x) = e^{-\mu\sqrt{1+\abs{x}^2}}$\,, $x\in\Rd$\,. Then $f\in \Lp^{p,\mu}((0, T)\times \Rd)$ if $fe_{\mu} \in \Lp^{p}((0, T)\times \Rd)$\,. Similarly, $\Sob^{1,2,p,\mu}((0, T)\times \Rd) = \{f\in \Lp^{p,\mu}((0, T)\times \Rd) \mid f, \frac{\partial f}{\partial t}, \frac{\partial f}{\partial x_i}, \frac{\partial^2 f}{\partial x_i \partial x_j}\in \Lp^{p,\mu}((0, T)\times \Rd) \}$ with natural norm (see \cite{BL84-book})
\begin{align*}
\norm{f}_{\Sob^{1,2,p,\mu}} =& \norm{\frac{\partial f}{\partial t}}_{\Lp^{p,\mu}((0, T)\times \Rd)} + \norm{f}_{\Lp^{p,\mu}((0, T)\times \Rd)} \\
& + \sum_{i}\norm{\frac{\partial f}{\partial x_i}}_{\Lp^{p,\mu}((0, T)\times \Rd)} + \sum_{i,j}\norm{\frac{\partial^2 f}{\partial x_i \partial x_j}}_{\Lp^{p,\mu}((0, T)\times \Rd)}\,. 
\end{align*} Also, we use the following convention $\norm{f}_{\Sob^{1,2,p,\mu}} = \norm{f}_{1,2,p,\mu}$\,.
\end{itemize}
\section{The Borkar Topology on Control Policies, Cost Criteria, and the Problem Statement}\label{PD} Let $\Act$ be a compact metric space and $\pV=\mathscr{P}(\Act)$ be the space of probability measures on  $\Act$ with topology of weak convergence. Let $$b : \Rd \times \Act \to  \Rd, $$ $$ \sigma : \Rd \to \RR^{d \times d},\, \sigma = [\sigma_{ij}(\cdot)]_{1\leq i,j\leq d},$$ be given functions. We consider a stochastic optimal control problem whose state is evolving according to a controlled diffusion process given by the solution of the following stochastic differential equation (SDE)
\begin{equation}\label{E1.1}
\D X_t \,=\, b(X_t,U_t) \D t + \upsigma(X_t) \D W_t\,,
\quad X_0=x\in\Rd.
\end{equation}
Where 
\begin{itemize}
\item
$W$ is a $d$-dimensional standard Wiener process, defined on a complete probability space $(\Omega, \sF, \mathbb{P})$.
\item 
 We extend the drift term $b : \Rd \times \pV \to  \Rd$ as follows:
\begin{equation*}
b (x,\mathrm{v}) = \int_{\Act} b(x,\zeta)\mathrm{v}(\D \zeta), 
\end{equation*}
for $\mathrm{v}\in\pV$.
\item
$U$ is a $\pV$ valued adapted process satisfying following non-anticipativity condition: for $s<t\,,$ $W_t - W_s$ is independent of
$$\sF_s := \,\,\mbox{the completion of}\,\,\, \sigma(X_0, U_r, W_r : r\leq s)\,\,\,\mbox{relative to} \,\, (\sF, \mathbb{P})\,.$$  
\end{itemize}
The process $U$ is called an \emph{admissible} control, and the set of all admissible controls is denoted by $\Uadm$ (see, \cite{BG90}). By a Markov control we mean an admissible control of the form $U_t = v(t,X_t)$ for some Borel measurable function $v:\RR_+\times\Rd\to\pV$. The space of all Markov controls is denoted by $\Um$\,. If the function $v$ is independent of $t$, i.e., $U_t = v(X_t)$ then $U$ or by an abuse of notation $v$ itself is called a stationary Markov control. The set of all stationary Markov controls is denoted by $\Usm$.

To ensure existence and uniqueness of strong solutions of \cref{E1.1}, we impose the following assumptions on the drift $b$ and the diffusion matrix $\upsigma$\,. 
\begin{itemize}
\item[\hypertarget{A1}{{(A1)}}]
\emph{Local Lipschitz continuity:\/}
The function
$\upsigma\,=\,\bigl[\upsigma^{ij}\bigr]\colon\RR^{d}\to\RR^{d\times d}$,
$b\colon\Rd\times\Act\to\Rd$ are locally Lipschitz continuous in $x$ (uniformly with respect to the other variables for $b$). In other words, for some constant $C_{R}>0$
depending on $R>0$, we have
\begin{equation*}
\abs{b(x,\zeta) - b(y, \zeta)}^2 + \norm{\upsigma(x) - \upsigma(y)}^2 \,\le\, C_{R}\,\abs{x-y}^2
\end{equation*}
for all $x,y\in \sB_R$ and $\zeta\in\Act$, where $\norm{\upsigma}\df\sqrt{\trace(\upsigma\upsigma\transp)}$\,. Also, we are assuming that $b$ is jointly continuous in $(x,\zeta)$.

\medskip
\item[\hypertarget{A2}{{(A2)}}]
\emph{Affine growth condition:\/}
$b$ and $\upsigma$ satisfy a global growth condition of the form
\begin{equation*}
\sup_{\zeta\in\Act}\, \langle b(x, \zeta),x\rangle^{+} + \norm{\upsigma(x)}^{2} \,\le\,C_0 \bigl(1 + \abs{x}^{2}\bigr) \qquad \forall\, x\in\RR^{d},
\end{equation*}
for some constant $C_0>0$.

\medskip
\item[\hypertarget{A3}{{(A3)}}]
\emph{Nondegeneracy:\/}
For each $R>0$, it holds that
\begin{equation*}
\sum_{i,j=1}^{d} a^{ij}(x)z_{i}z_{j}
\,\ge\,C^{-1}_{R} \abs{z}^{2} \qquad\forall\, x\in \sB_{R}\,,
\end{equation*}
and for all $z=(z_{1},\dotsc,z_{d})\transp\in\RR^{d}$,
where $a\df \frac{1}{2}\upsigma \upsigma\transp$.
\end{itemize}

\subsection{The Borkar Topology on Control Policies}\label{B-topo}
We now introduce the Borkar topology on stationary or Markov controls \cite{Bor89}
\begin{itemize}
\item[•]\emph{Topology of Stationary Policies:}
From \cite[Section~2.4]{ABG-book}, we have that the set $\Usm$ is metrizable with compact metric.
\begin{definition}[Borkar topology of stationary Markov policies]\label{DefBorkarTopology1A}
A sequence $v_n\to v$ in $\Usm$ if and only if
\begin{equation}\label{BorkarTopology}
\lim_{n\to\infty}\int_{\Rd}f(x)\int_{\Act}g(x,\zeta)v_{n}(x)(\D \zeta)\D x = \int_{\Rd}f(x)\int_{\Act}g(x,\zeta)v(x)(\D \zeta)\D x
\end{equation}
for all $f\in L^1(\Rd)\cap L^2(\Rd)$ and $g\in \cC_b(\Rd\times \Act)$ (for more details, see \cite[Lemma~2.4.1]{ABG-book}, \cite{Bor89})\,.
\end{definition}
\item[•]\emph{Topology of Markov Policies:} In the proof of \cite[Theorem~3.1, Lemma~3.1]{Bor89}, replacing $A_n$ by $\hat{A}_n = A_n\times [0,n]$ and following the arguments as in the proof of \cite[Theorem~3.1, Lemma~3.1]{Bor89}, we have the following topology on the space of Markov policies $\Um$\,. 
\begin{definition}[Borkar topology of Markov policies]\label{BKTP1}
A sequence $v_n\to v$ in $\Um$ if and only if
\begin{equation}\label{BorkarTopologyM}
\lim_{n\to\infty}\int_{0}^{\infty}\int_{\Rd}f(t,x)\int_{\Act}g(x,t,\zeta)v_{n}(t,x)(\D \zeta)\D x \D t = \int_{0}^{\infty}\int_{\Rd}f(t,x)\int_{\Act}g(x,t,\zeta)v(x,t)(\D \zeta)\D x \D t 
\end{equation}
for all $f\in L^1(\Rd\times [0, \infty))\cap L^2(\Rd\times [0, \infty))$ and $g\in \cC_b(\Rd\times [0, \infty)\times \Act)$\,.
\end{definition}
\end{itemize}
It is well known that under the hypotheses \hyperlink{A1}{{(A1)}}--\hyperlink{A3}{{(A3)}}, for any admissible control \cref{E1.1} has a unique weak solution \cite[Theorem~2.2.11]{ABG-book}, and under any stationary Markov strategy \cref{E1.1} has a unique strong solution which is a strong Feller (therefore strong Markov) process \cite[Theorem~2.2.12]{ABG-book}.

\subsection{Cost Criteria}
Let $c\colon\Rd\times\Act \to \RR_+$ be the \emph{running cost} function. We assume that $c$ is bounded, jointly continuous in $(x, \zeta)$ and locally Lipschitz continuous in its first argument uniformly with respect to $\zeta\in\Act$. We extend $c\colon\Rd\times\pV \to\RR_+$ as follows: for $\pv \in \pV$
\begin{equation*}
c(x,\pv) := \int_{\Act}c(x,\zeta)\pv(\D\zeta)\,.
\end{equation*}
In this article, we consider the problem of minimizing finite horizon cost, $\alpha$-discounted cost and ergodic cost, respectively:
\subsubsection{Finite Horizon Cost}
For $U\in \Uadm$, the associated \emph{finite horizon cost} is given by
\begin{equation}\label{FiniteCost1}
\cJ_{T}(x, U) = \Exp_x^{U}\left[\int_0^{T} c(X_s, U_s) \D{s} + H(X_T)\right]\,.
\end{equation} and the optimal value is defined as
\begin{equation}\label{FiniteCost1Opt}
\cJ_{T}^*(x) \,\df\, \inf_{U\in \Uadm}\cJ_{T}(x, U)\,.
\end{equation}
Then a policy $U^*\in \Uadm$ is said to be optimal if we have 
\begin{equation}\label{FiniteCost1Opt1}
\cJ_{T}(x, U^*) = \cJ_{T}^*(x)\,.
\end{equation}
\subsubsection{Discounted Cost Criterion} 
For $U \in\Uadm$, the associated \emph{$\alpha$-discounted cost} is given by
\begin{equation}\label{EDiscost}
\cJ_{\alpha}^{U}(x, c) \,\df\, \Exp_x^{U} \left[\int_0^{\infty} e^{-\alpha s} c(X_s, U_s) \D s\right],\quad x\in\Rd\,,
\end{equation} where $\alpha > 0$ is the discounted factor
and $X(\cdot)$ is the solution of \cref{E1.1} corresponding to $U\in\Uadm$ and $\Exp_x^{U}$ is the expectation with respect to the law of the process $X(\cdot)$ with initial condition $x$. The controller tries to minimize \cref{EDiscost} over his/her admissible policies $\Uadm$\,. Thus, a policy $U^{*}\in \Uadm$ is said to be optimal if for all $x\in \Rd$ 
\begin{equation}\label{OPDcost}
\cJ_{\alpha}^{U^*}(x, c) = \inf_{U\in \Uadm}\cJ_{\alpha}^{U}(x, c) \,\,\, (\,=:\, \,\, V_{\alpha}(x))\,,
\end{equation} where $V_{\alpha}(x)$ is called the optimal value.
\subsubsection{Ergodic Cost Criterion}
For $U\in \Uadm$, the associated \emph{ergodic cost} is given by
\begin{equation}\label{ErgCost1}
\sE_{x}(c, U) = \limsup_{T\to \infty}\frac{1}{T}\Exp_x^{U}\left[\int_0^{T} c(X_s, U_s) \D{s}\right]\,.
\end{equation} and the optimal value is defined as
\begin{equation}\label{ErgCost1Opt}
\sE^*(c) \,\df\, \inf_{x\in\Rd}\inf_{U\in \Uadm}\sE_{x}(c, U)\,.
\end{equation}
Then a policy $U^*\in \Uadm$ is said to be optimal if we have 
\begin{equation}\label{ErgCost1Opt1}
\sE_{x}(c, U^*) = \sE^*(c)\,.
\end{equation}

\subsubsection{Control up to an Exit Time}\label{exitTimeSection}
For each $U\in\Uadm$ the associated cost is given as
\begin{equation*}
\hat{\cJ}_{e}^{U}(x) \,\df \, \Exp_x^{U} \left[\int_0^{\tau(O)} e^{-\int_{0}^{t}\delta(X_s, U_s) \D s} c(X_t, U_t) \D t + e^{-\int_{0}^{\tau(O)}\delta(X_s, U_s) \D s}h(X_{\tau(O)})\right],\quad x\in\Rd\,,
\end{equation*}
where $O\subset \Rd$ is a smooth bounded domain, $\tau(O) \,\df\,  \inf\{t \geq 0: X_t\notin O\}$, $\delta(\cdot, \cdot): \bar{O}\times\Act\to [0, \infty)$ is the discount function and $h:\bar{O}\to \RR_+$ is the terminal cost function. The optimal value is defined as \[\hat{\cJ}_{e}^{*}(x)=\inf_{U\in \Uadm}\hat{\cJ}_{e}^{U}(x),\]
We assume that $\delta\in \cC(\bar{O}\times \Act)$, $h\in\Sob^{2,p}(O)$. 


\subsection{Problems Studied}

The main purpose of this manuscript will be to address the following problems:
\begin{itemize}
\item[•]\textbf{Continuity of finite and infinite horizon costs.} Suppose $\{v_n\}_{n\in \NN}$ is a sequence of control policies which converge to another control policy $v$ in some sense (in particular, under Borkar topology, see Subsection~\ref{B-topo}). Does this imply that
\begin{itemize}
\item[•]\emph{for finite horizon cost:} $\cJ_{T}(x, v_n)\to \cJ_{T}(x, v)$\,?
\item[•]\emph{for discounted cost:} $\cJ_{\alpha}^{v_n}(x, c)\to \cJ_{\alpha}^{v}(x, c)$ \,?
\item[•]\emph{for ergodic cost:} $\sE_{x}(c, v_n)\to \sE_{x}(c, v)$ \,?
\item[•] \emph{for cost up to an exit time:} $\hat{\cJ}_{e}^{v_n}(x) \to \hat{\cJ}_{e}^{v}(x)$ \,?
\end{itemize}
\item[•]\textbf{Near optimality of quantized policies.} For any given $\epsilon > 0$, whether it is possible to construct a quantized (finite action/ piecewise constant) policy $v_{\epsilon}$ such that
\begin{itemize}
\item[•]\emph{for finite horizon cost:} $\cJ_{T}(x, v_{\epsilon})\leq \cJ_{T}^*(x) + \epsilon$\,?
\item[•]\emph{for discounted cost:} $\cJ_{\alpha}^{v_{\epsilon}}(x, c)\leq V_{\alpha}(x) + \epsilon$ \,?
\item[•]\emph{for ergodic cost:} $\sE_{x}(c, v_{\epsilon})\leq \sE^*(c) + \epsilon$ \,?
\item[•] \emph{for cost up to an exit time:} $\hat{\cJ}_{e}^{v_{\epsilon}}(x) \leq \hat{\cJ}_{e}^{*}(x) + \epsilon$ \,?
\end{itemize}
\end{itemize}
In this manuscript, we have shown that under a mild set of assumptions the answers to the above mentioned questions are affirmative. For the finite horizon case, we also study the time-discretization approximations as a further implication of our analysis.

Let us introduce a parametric family of elliptic operator, which will be useful in our analysis\,. With $\zeta\in \Act$ treated as a parameter, we define a family of operators $\sL_{\zeta}$ mapping
$\cC^2(\Rd)$ to $\cC(\Rd)$ by
\begin{equation}\label{E-cI}
\sL_{\zeta} f(x) \,\df\, \trace\bigl(a(x)\grad^2 f(x)\bigr) + \,b(x,\zeta)\cdot \grad f(x)\,, 
\end{equation}
where $f\in \cC^2(\Rd)\cap\cC_b(\Rd)$\, and for $\pv \in\pV$ we extend $\sL_{\zeta}$ as follows:
\begin{equation}\label{EExI}
\sL_\pv f(x) \,\df\, \int_{\Act} \sL_{\zeta} f(x)\pv(\D \zeta)\,.
\end{equation}Also, for each $v \in\Usm$, we define
\begin{equation}\label{Efixstra}
\sL_{v} f(x) \,\df\, \trace(a\grad^2 f(x)) + b(x,v(x))\cdot\grad f(x)\,.
\end{equation}

\section{Continuity of Expect Cost under Various Criteria in Control Policies under the Borkar Topology}

\subsection{Continuity for Discounted Cost/Cost upto an Exit Time}\label{CDiscCost}
Since the proof techniques are almost similar, in this section, we analyze the continuity of both discounted cost as well as the cost upto an exit time with respect to the policies in the space of stationary policies under Borkar topology (see Definition~\ref{DefBorkarTopology1A}), i.e., we show that the maps $v\to \cJ_{\alpha}^{v}$ and $v\to \hat{\cJ}_{e}^{v}$ are continuous on $\Usm$\,.  
\subsubsection{Continuity of Discounted Cost}
Now we prove the continuity of the discounted cost as a function of the control policies\,.
\begin{theorem}\label{T1.1}
Suppose Assumptions (A1)-(A3) hold. Then the map $v\mapsto \cJ_{\alpha}^{v}(x, c)$ from $\Usm$ to $\RR$ is continuous. 
\end{theorem}
\begin{proof}
Let $\{v_n\}_n$ be a sequence in $\Usm$ such that $v_n\to v$ in $\Usm$\,. It known that $\cJ_{\alpha}^{v_n}(x, c)$ is a solution to the Poisson's equation (see, \cite[Lemma~A.3.7]{ABG-book})
\begin{equation}\label{ET1.1A}
 \sL_{v_n}\cJ_{\alpha}^{v_n}(x, c) - \alpha \cJ_{\alpha}^{v_n}(x, c) = - c(x, v_n(x))\,.
\end{equation}
Now by standard elliptic p.d.e. estimates as in \cite[Theorem~9.11]{GilTru}, for any $p\geq d+1$ and $R >0$, we deduce that
\begin{equation}\label{ET1.1B}
\norm{\cJ_{\alpha}^{v_n}(x, c)}_{\Sob^{2,p}(\sB_R)}
\,\le\, \kappa_1\bigl(\norm{\cJ_{\alpha}^{v_n}(x, c)}_{L^p(\sB_{2R})} + \norm{c(x, v_n(x))}_{L^p(\sB_{2R})}\bigr)\,,
\end{equation}
for some positive constant $\kappa_1$ which is independent of $n$\,. Since 
\begin{equation*}
\norm{c}_{\infty} \,\df\, \sup_{(x,u)\in\Rd\times\Act} c(x,u) \leq M \, <\,\infty \,, \quad \text{and}\quad \cJ_{\alpha}^{v_n}(x, c) \leq \frac{\norm{c}_{\infty}}{\alpha}\,,
\end{equation*} from \cref{ET1.1A} we obtain
\begin{equation}\label{ETC1.3B}
\norm{\cJ_{\alpha}^{v_n}(x, c)}_{\Sob^{2,p}(\sB_R)}
\,\le\, \kappa_1 M\bigl(\frac{|\sB_{2R}|^{\frac{1}{p}}}{\alpha} + |\sB_{2R}|^{\frac{1}{p}}\bigr)\,.
\end{equation}We know that for $1< p < \infty$, the space $\Sob^{2,p}(\sB_R)$ is reflexive and separable, hence, as a corollary of Banach Alaoglu theorem, we have that every bounded sequence in $\Sob^{2,p}(\sB_R)$ has a weakly convergent subsequence (see, \cite[Theorem~3.18.]{HB-book}). Also, we know that for $p\geq d+1$ the space $\Sob^{2,p}(\sB_R)$ is compactly embedded in $\cC^{1, \beta}(\bar{\sB}_R)$\,, where $\beta < 1 - \frac{d}{p}$ (see \cite[Theorem~A.2.15 (2b)]{ABG-book}), which implies that every weakly convergent sequence in $\Sob^{2,p}(\sB_R)$ will converge strongly in $\cC^{1, \beta}(\bar{\sB}_R)$\,. Thus, in view of estimate \cref{ETC1.3B}, by a standard diagonalization argument and Banach Alaoglu theorem, we can extract a subsequence $\{V_{\alpha}^{n_k}\}$ such that for some $V_{\alpha}^*\in \Sobl^{2,p}(\Rd)$
\begin{equation}\label{ET1.1C}
\begin{cases}
\cJ_{\alpha}^{v_{n_k}}(x, c)\to & V_{\alpha}^*\quad \text{in}\quad \Sobl^{2,p}(\Rd)\quad\text{(weakly)}\\
\cJ_{\alpha}^{v_{n_k}}(x, c)\to & V_{\alpha}^*\quad \text{in}\quad \cC^{1, \beta}_{loc}(\Rd) \quad\text{(strongly)}\,.
\end{cases}       
\end{equation} 
In the following we will show that $V_{\alpha}^* = \cJ_{\alpha}^{v}(x, c)$. Note that 
\begin{align*}
b(x,v_{n_k}(x))\cdot \grad \cJ_{\alpha}^{v_{n_k}}(x, c) - b(x,v(x))\cdot \grad V_{\alpha}^*(x) = & b(x,v_{n_k}(x))\cdot \grad \left(\cJ_{\alpha}^{v_{n_k}}(x, c) - V_{\alpha}^*\right)(x) \\
& + \left(b(x,v_{n_k}(x)) - b(x,v(x))\right)\cdot \grad V_{\alpha}^*(x)\,.
\end{align*}
Since $\cJ_{\alpha}^{v_{n_k}}(x, c)\to V_{\alpha}^*$ in $\cC^{1, \beta}_{loc}(\Rd)$ and $b$ is locally bounded, on any compact set $b(x,v_{n_k}(x))\cdot \grad \left(\cJ_{\alpha}^{v_{n_k}}(x, c) - V_{\alpha}^*\right)(x)\to 0$ strongly. Also, since $\grad V_{\alpha}^*\in \cC^{1, \beta}_{loc}(\Rd)$, in view of the topology of $\Usm$, for any $\phi\in\cC_c^{\infty}(\Rd)$ we have 
\begin{equation*}
\lim_{n\to\infty}\int_{\Rd}b(x,v_{n_k}(x))\cdot \grad V_{\alpha}^*(x)\phi(x)\D x = \int_{\Rd}b(x,v(x))\cdot \grad V_{\alpha}^*(x)\phi(x)\D x\,.
\end{equation*}
Hence, as $k\to \infty$, we obtain 
\begin{equation}\label{ET1.1D}
b(x,v_{n_k}(x))\cdot \grad \cJ_{\alpha}^{v_{n_k}}(x, c) + c(x, v_{n_k}(x)) \to b(x,v(x))\cdot \grad V_{\alpha}^*(x) + c(x, v(x))\quad\text{weakly}\,.
\end{equation}
Now, multiplying by a test function $\phi\in \cC_{c}^{\infty}(\Rd)$, from \cref{ET1.1A}, it follows that
\begin{align*}
\int_{\Rd}\trace\bigl(a(x)\grad^2 \cJ_{\alpha}^{v_{n_k}}(x, c)\bigr)\phi(x)\D x + \int_{\Rd}\{b(x,v_{n_k}(x))\cdot \grad \cJ_{\alpha}^{v_{n_k}}(x, c) + & c(x, v_{n_k}(x))\}\phi(x)\D x \\
&= \alpha\int_{\Rd} \cJ_{\alpha}^{v_{n_k}}(x, c)\phi(x)\D x\,.
\end{align*}
Hence, using \cref{ET1.1C}, \cref{ET1.1D}, and letting $k\to\infty$ (in the sense of distributions), we obtain 
\begin{equation}\label{ET1.1E}
\int_{\Rd}\trace\bigl(a(x)\grad^2 V_{\alpha}^*(x)\bigr)\phi(x)\D x + \int_{\Rd} \{b(x,v(x))\cdot \grad V_{\alpha}^*(x) + c(x, v(x))\}\phi(x)\D x = \alpha\int_{\Rd} V_{\alpha}^*(x)\phi(x)\D x\,.
\end{equation} Since $\phi\in \cC_{c}^{\infty}(\Rd)$ is arbitrary and $V_{\alpha}^*\in \Sobl^{2,p}(\Rd)$ from \cref{ET1.1E}, we deduce that
the function $V_{\alpha}^*\in \Sobl^{2,p}(\Rd)\cap \cC_{b}(\Rd)$ satisfies
\begin{equation}\label{ET1.1F}
\trace\bigl(a(x)\grad^2 V_{\alpha}^{*}(x)\bigr) + b(x,v(x))\cdot \grad V_{\alpha}^{*}(x) + c(x, v(x)) = \alpha V_{\alpha}^{*}(x)\,.
\end{equation}
Let $X$ be the solution of the SDE \cref{E1.1} corresponding to $v$. Now applying It$\hat{\rm o}$-Krylov formula, we obtain the following
\begin{align*}
& \Exp_x^{v}\left[ e^{-\alpha T}V_{\alpha}^{*}(X_{T})\right] - V_{\alpha}^{*}(x)\nonumber\\ 
& \,=\,\Exp_x^{v}\left[\int_0^{T} e^{-\alpha s}\{\trace\bigl(a(X_s)\grad^2 V_{\alpha}^{*}(X_s)\bigr) + b(X_s, v(X_s))\cdot \grad V_{\alpha}^{*}(X_s) - \alpha V_{\alpha}^{*}(X_s))\} \D{s}\right] \,.
\end{align*}
Hence, by \cref{ET1.1F}, we get
\begin{align}\label{ET1.1G}
\Exp_x^{v}\left[ e^{-\alpha T} V_{\alpha}^{*}(X_{T})\right] - V_{\alpha}^{*}(x) \,=\,- \Exp_x^{v}\left[\int_0^{T} e^{-\alpha s}c(X_s, v(X_s))\D{s}\right] \,.
\end{align}
Since $V_{\alpha}^{*}$ is bounded and
$$\Exp_x^{v}\left[ e^{-\alpha T} V_{\alpha}^{*}(X_{T})\right] = e^{-\alpha T}\Exp_x^{v}\left[ V_{\alpha}^{*}(X_{T})\right],$$
letting $T\to\infty$, it follows that
\begin{equation*}
\lim_{T\to\infty}\Exp_x^{v}\left[ e^{-\alpha T} V_{\alpha}^{*}(X_{T})\right] = 0\,.
\end{equation*}
Thus, letting $T \to \infty$ by monotone convergence theorem, from \cref{ET1.1G}, we obtain 
\begin{align}\label{ET1.1H}
 V_{\alpha}^{*}(x) \,=\, \Exp_x^{v}\left[\int_0^{\infty} e^{-\alpha s}c(X_s, v(X_s)) \D{s}\right] = \cJ_{\alpha}^{v}(x, c)\,.
\end{align}This completes the proof.
\end{proof}
\subsubsection{Continuity of Cost upto an Exit Time} Following the proof technique of Theorem~\ref{T1.1}, now we show that the cost upto an exit time (defined in Subsection~\ref{exitTimeSection}) is continuous as a function of the control policies\,.  
\begin{theorem}\label{T1.1Exit}
Suppose Assumptions (A1)-(A3) hold. Then the map $v\mapsto \hat{\cJ}_{e}^{v}(x)$ from $\Usm$ to $\RR$ is continuous\,. 
\end{theorem}
\begin{proof}
Let $\{v_n\}_n$ be a sequence in $\Usm$ such that $v_n\to v$ in $\Usm$\,. From \cite[Theorem~9.15]{GilTru}, it follows that there exist a unique function $\psi_n(x)\in \Sob^{2,p}(O)$ satisfying the following Poisson's equation
\begin{equation}\label{T1.1ExitA}
 \sL_{v_n}\psi_{n}(x) - \delta(x, v_n(x)) \psi_{n}(x) + c(x, v_n(x)) = 0\quad \text{with}\quad \psi_n = h\,\,\, \text{on}\,\,\partial{O}\,.
\end{equation} Applying It$\hat{\rm o}$-Krylov formula, one can show that $\psi_n(x) = \hat{\cJ}_{e}^{v_n}(x)$ (this stochastic representation also ensures the uniqueness of the solution of \cref{T1.1ExitA} )\,.

Now following the argument as in Theorem~\ref{T1.1}, by standard elliptic p.d.e. estimates \cite[Theorem~9.11]{GilTru}, we deduce that there exists $\psi(x)\in \Sob^{2,p}(O)$ such that $\psi_n \to \psi$ weakly in $\Sob^{2,p}(O)$\,. Thus, closely following the proof of Theorem~\ref{T1.1}, letting $n\to \infty$\,, from \cref{T1.1ExitA} it follows that
\begin{equation}\label{T1.1ExitB}
 \sL_{v}\psi(x) - \delta(x, v(x)) \psi(x) + c(x, v(x)) = 0\quad \text{with}\quad \psi = h\,\,\, \text{on}\,\,\partial{O}\,.
\end{equation} Again, by It$\hat{\rm o}$-Krylov formula, using \cref{T1.1ExitB} we deduce that $\psi(x) = \hat{\cJ}_{e}^{v}(x)$\,. This completes the proof of the theorem\,.
\end{proof}

\subsection{Continuity for Ergodic Cost}\label{CErgoCost}
In this section we study the continuity of the ergodic costs with respect to policies under Borkar topology in the space of stationary Markov policies. We will study this problem under two sets of assumptions: the first is so called near-monotonicity assumption on the running cost function and other one is Lyapunov stability assumption on the system. Our proof strategies will be slightly different under these two setups: In the former we will build on regularity properties of invariant probability measures, in the latter we will build more directly on regularity properties of solutions to HJB equations\,.
\subsubsection{Under a near-monotonicity assumption}\label{NearMonotone}    
We assume that the running cost function $c$ is near-monotone with respect to $\sE^*(c)$, i.e.,
\begin{itemize}
\item[\hypertarget{A4}{{(A4)}}] It holds that
\begin{equation}\label{ENearmonot}
\liminf_{\norm{x}\to\infty}\inf_{\zeta\in \Act} c(x,\zeta) > \sE^*(c)\,. 
\end{equation}
\end{itemize}
This condition penalizes the escape of probability mass to infinity. Since our running cost $c$ is bounded it is easy to see that $\sE^*(c) \leq \norm{c}_{\infty}$\,. It is known that under \cref{ENearmonot}, optimal control exists in the space of stable stationary Markov controls (see, \cite[Theorem~3.4.5]{ABG-book}). 
 
First, we prove that for each stable stationary Markov policy $v\in \Usm$ the associated Poisson's equation admits a unique solution in a certain function space. This uniqueness result will be useful in establishing the continuity and near optimality of quantized policies. For the following supporting result, we closely follow \cite{ABG-book}\,.
\begin{theorem}\label{NearmonotPoisso}
Suppose that Assumptions (A1) - (A4) hold. Let $v\in\Usm$ be a stable control with unique invariant measure $\eta_{v}$, such that 
\begin{equation}\label{ENearmonotPoisso1}
\liminf_{\norm{x}\to\infty}\inf_{\zeta\in \Act} c(x,\zeta) > \inf_{x\in\Rd}\sE_x(c, v)\,. 
\end{equation} Then, there exists a unique pair $(V^v, \rho_v)\in \Sobl^{2,p}(\Rd)\times \RR$, \, $1< p < \infty$, with $V^v(0) = 0$, $\inf_{\Rd} V^v > -\infty$ and $\rho_v = \int_{\Rd}\int_{\Act} c(x, u)v(x)(\D{u})\eta_{v}(\D{x})$, satisfying
\begin{equation}\label{EErgonearPoisso1A}
\rho_v = \left[\sL_{v}V^v(x) + c(x, v(x))\right]
\end{equation}
Moreover, we have
\begin{itemize}
\item[(i)]$\rho_v = \inf_{\Rd}\sE_x(c, v)$\,.
\item[(ii)] for all $x\in \Rd$
\begin{equation}\label{EErgonearPoisso1B}
V^v(x) \,=\, \lim_{r\downarrow 0}\Exp_{x}^v\left[\int_{0}^{\uuptau_{r}} \left( c(X_t, v(X_t)) - \rho_v\right)\D t\right]\,.
\end{equation}
\end{itemize} 
\end{theorem}
\begin{proof}
Since $c$ is bounded, we have $\left(\rho^{v} \,\df\,\right) \int_{\Rd}\int_{\Act} c(x, u)v(x)(\D{u})\eta_{v}(\D{x}) \leq \norm{c}_{\infty}$\,. In view of \cref{ENearmonotPoisso1}, by writing $\rho^v =\alpha \int J^v_{\alpha}(x, c)\eta_{v}(\D x)$ from \cite[Lemma~3.6.1]{ABG-book}, we have 
\begin{equation}\label{EErgonearPoisso1C}
\inf_{\kappa(\rho^v)}\cJ_{\alpha}^{v}(x, c) = \inf_{\Rd}\cJ_{\alpha}^{v}(x, c) \leq \frac{\rho^{v}}{\alpha}\,,
\end{equation} where $\kappa(\rho^v) \,\df\, \{x\in \Rd\mid \min_{\zeta\in\Act}c(x,\zeta) \leq \rho^v\}$ and $\cJ_{\alpha}^{v}(x, c)$ is the $\alpha$-discounted cost defined as in \cref{EDiscost}. As earlier, we have that $\cJ_{\alpha}^{v}(x, c)$ is a solution to the Poisson's equation (see, \cite[Lemma~A.3.7]{ABG-book})
\begin{equation}\label{EErgonearPoisso1D}
 \sL_{v}\cJ_{\alpha}^{v}(x, c) - \alpha \cJ_{\alpha}^{v}(x, c) = - c(x, v(x))\,.
\end{equation} Since $x\to\min_{\zeta\in\Act}c(x,\zeta)$ is continuous, we have $\kappa(\rho^v)$ is closed and \cref{ENearmonotPoisso1} implies that it is bounded. Therefore $\kappa(\rho^v)$ is compact and hence for some $R_0>0$, we have $\kappa(\rho^v)\subset \sB_{R_{0}}$\,. This gives us $\inf_{\sB_{R_{0}}} \cJ_{\alpha}^{v}(x, c) = \inf_{\Rd}\cJ_{\alpha}^{v}(x, c)$\,.

Thus, following the arguments as in \cite[Lemma~3.6.3]{ABG-book}, we deduce that for each $R> R_0$ there exist constants $\Tilde{C}_{2}(R), \Tilde{C}_{2}(R, p)$ depending only on $d, R_0$ such that
\begin{equation}\label{EErgonearPoisso1E}
\left(\osc_{\sB_{2R}} \cJ_{\alpha}^{v}(x, c) := \right)\sup_{\sB_{2R}} \cJ_{\alpha}^{v}(x, c) - \inf_{\sB_{2R}} \cJ_{\alpha}^{v}(x, c) \leq \Tilde{C}_{2}(R)\left(1 + \alpha\inf_{\sB_{R_0}}\cJ_{\alpha}^{v}(x, c) \right)\,,
\end{equation}
\begin{equation}\label{EErgonearPoisso1F}
\norm{\cJ_{\alpha}^{v}(\cdot, c) - \cJ_{\alpha}^{v}(0, c)}_{\Sob^{2,p}(\sB_R)}\leq \Tilde{C}_{2}(R, p) \left(1 + \alpha\inf_{\sB_{R_0}}\cJ_{\alpha}^{v}(x, c) \right)\,.
\end{equation}
Hence, by following the arguments as in \cite[Lemma~3.6.6]{ABG-book}, we conclude that there exists $(V^{v}, \Tilde{\rho}^v)\in \Sobl^{2,p}(\Rd)\times \RR$ such that along a subsequence (as $\alpha\to 0$), $\cJ_{\alpha}^{v}(\cdot, c) - \cJ_{\alpha}^{v}(0, c) \to V^{v}(\cdot)$ and $\alpha\cJ_{\alpha}^{v}(0, c) \to \Tilde{\rho}_{v}$ and the pair $(V^{v}, \Tilde{\rho}_v)$ satisfies
\begin{equation}\label{EErgonearPoisso1G}
 \sL_{v}V^{v}(x) + c(x, v(x)) = \Tilde{\rho}_{v}\,.
\end{equation} 
We will show that the subsequential limits are unique\,.

From \cref{EErgonearPoisso1C}, we get $\Tilde{\rho}_{v} \leq \rho^{v}$. Now, in view of estimates \cref{EErgonearPoisso1C} and \cref{EErgonearPoisso1F}, it is easy to see that
\begin{equation}\label{EErgonearPoisso1H}
\norm{V^{v}}_{\Sob^{2,p}(\sB_R)}\leq \Tilde{C}_{2}(R, p) \left(1 + M \right)\,.
\end{equation} Also, for each $x\in \Rd$, we have
\begin{align}\label{EErgonearPoisso1HLower}
V^{v}(x) &= \lim_{\alpha\to 0}\left(\cJ_{\alpha}^{v}(x, c) - \cJ_{\alpha}^{v}(0, c)\right) \geq \liminf_{\alpha\to 0} \left(\cJ_{\alpha}^{v}(x, c) - \inf_{\Rd}\cJ_{\alpha}^{v}(x, c) + \inf_{\Rd}\cJ_{\alpha}^{v}(x, c) - \cJ_{\alpha}^{v}(0, c)\right) \nonumber\\
&\geq - \limsup_{\alpha\to \infty} \left(\cJ_{\alpha}^{v}(0, c) - \inf_{\Rd}\cJ_{\alpha}^{v}(x, c)\right) + \liminf _{\alpha\to \infty} \left(\cJ_{\alpha}^{v}(x, c) - \inf_{\Rd}\cJ_{\alpha}^{v}(x, c)\right)\nonumber\\
&\geq - \limsup_{\alpha\to \infty} \left(\cJ_{\alpha}^{v}(0, c) - \inf_{\sB_{R_0}}\cJ_{\alpha}^{v}(x, c)\right) + \liminf _{\alpha\to \infty} \left(\cJ_{\alpha}^{v}(x, c) - \inf_{\Rd}\cJ_{\alpha}^{v}(x, c)\right)\nonumber\\
&\geq - \limsup_{\alpha\to \infty} \left(\osc_{\sB_{R_0}} \cJ_{\alpha}^{v}(x, c)\right);\quad \left(\text{since}\,\,\, \cJ_{\alpha}^{v}(x, c) - \inf_{\Rd}\cJ_{\alpha}^{v}(x, c) \geq 0 \right)\,,
\end{align} where in the third inequality we have used the fact that $\inf_{\sB_{R_0}} \cJ_{\alpha}^{v}(x, c) = \inf_{\Rd} \cJ_{\alpha}^{v}(x, c)$\,. Thus, from \cref{EErgonearPoisso1E}, we deduce that
\begin{equation}\label{EErgonearPoisso1I}
V^{v}\geq -\Tilde{C}_{2}(R_0) \left(1 + M \right)\,. 
\end{equation} This shows that $\inf_{\Rd} V^{v} > -\infty$\,.

Now, applying It$\hat{\rm o}$-Krylov formula and using \cref{EErgonearPoisso1G} we obtain
\begin{align*}
 \Exp_x^{v}\left[V^{v}\left(X_{T\wedge \uptau_{R}}\right)\right] - V^v(x)\,=\, \Exp_x^{v}\left[\int_0^{T\wedge \uptau_{R}} \left(\Tilde{\rho}_{v} - c(X_t, v(X_t))\right) \D{t}\right]\,.
\end{align*} This implies
\begin{align*}
 \inf_{y\in\Rd}V^{v}(y) - V^v(x)\,\leq\, \Exp_x^{v}\left[\int_0^{T\wedge \uptau_{R}} \left(\Tilde{\rho}_{v} - c(X_t, v(X_t))\right) \D{t}\right]\,.
\end{align*}Since $v$ is stable, letting $R\to \infty$, we get
\begin{align*}
 \inf_{y\in\Rd}V^{v}(y) - V^v(x)\,\leq\, \Exp_x^{v}\left[\int_0^{T} \left(\Tilde{\rho}_{v} - c(X_t, v(X_t))\right) \D{t}\right]\,.
\end{align*}Now dividing both sides of the above inequality by $T$ and letting $T\to \infty$, it follows that
\begin{align*}
 \limsup_{T\to \infty}\frac{1}{T}\Exp_x^{v}\left[\int_0^{T} c(X_t, v(X_t)) \D{t}\right] \,\leq\, \Tilde{\rho}_{v}\,.
\end{align*} Thus, $\rho^v \leq \Tilde{\rho}_{v}$. This indeed implies that $\rho^v = \Tilde{\rho}_{v}$\,. The representation \cref{EErgonearPoisso1B} of $V^v$ follows by closely mimicking the argument of \cite[Lemma~3.6.9]{ABG-book}. Therefore, we have a solution pair $(V^v, \rho_v)$ to \cref{EErgonearPoisso1A} satisfying (i) and (ii). 

Next we want to prove that the solution pair is unique. To this end, let $(\hat{V}^v, \hat{\rho}_v)\in \Sobl^{2,p}(\Rd)\times \RR$, \, $1< p < \infty$, with $\hat{V}^v(0) = 0$, $\inf_{\Rd} \hat{V}^v > -\infty$ and $\hat{\rho}_v = \int_{\Rd}\int_{\Act} c(x, u)v(x)(\D{u})\eta_{v}(\D{x})$, satisfying
\begin{equation}\label{EErgonearPoisso1J}
\hat{\rho}_v = \left[\sL_{v}\hat{V}^v(x) + c(x, v(x))\right]
\end{equation}   
Since $\hat{V}^v$ is bounded from below, applying It$\hat{\rm o}$-Krylov formula and using \cref{EErgonearPoisso1J} we get
\begin{align}\label{EErgonearPoisso1L}
\limsup_{T\to \infty}\frac{1}{T}\Exp_x^{v}\left[\int_0^{T} c(X_t, v(X_t)) \D{t}\right] \,\leq\, \hat{\rho}_{v}
\end{align} Hence, from \cref{EErgonearPoisso1L}, it follows that
\begin{align}\label{EErgonearPoisso1M}
\hat{\rho}_{v} = \int_{\Rd}\int_{\Act} c(x, u)v(x)(\D{u})\eta_{v}(\D{x}) \leq \limsup_{T\to \infty}\frac{1}{T}\Exp_x^{v}\left[\int_0^{T} c(X_t, v(X_t)) \D{t}\right] \,\leq\, \hat{\rho}_{v}
\end{align} This implies that $\hat{\rho}^{v} = \rho_{v}$\,. Now, applying It$\hat{\rm o}$-Krylov formula and using \cref{EErgonearPoisso1J}, we obtain
\begin{align}\label{EErgonearPoisso1N}
\hat{V}^v(x)\,=\, \Exp_x^{v}\left[\int_0^{\uuptau_{r}\wedge \uptau_{R}} \left(c(X_t, v(X_t)) - \hat{\rho}_{v}\right) \D{t} + \hat{V}^{v}\left(X_{\uuptau_{r}\wedge \uptau_{R}}\right)\right]\,.
\end{align} Since $v$ is stable and $\hat{V}^v$ is bounded from below, for all $x\in \Rd$ we obtain
\begin{equation*}
\liminf_{R\to\infty}\Exp_x^{v}\left[\hat{V}^{v}\left(X_{\uptau_{R}}\right)\Ind_{\{\uuptau_{r}\geq \uptau_{R}\}}\right]\geq \liminf_{R\to\infty}\left(\inf_{\Rd}\hat{V}^{v}\right)\mathbb{P}_{x}\left(\uuptau_{r}\geq \uptau_{R}\right) = 0\,.
\end{equation*} In the above we have used the fact that $\uptau_{R}\to\infty$ as $R\to \infty$ and $\Exp_x^{v}\left[\uuptau_{r}\right] < \infty$\,.

Again, since $v$ is stable we have $\Exp_x^{v}\left[\uuptau_{r}\right] < \infty$ (see \cite[Theorem~2.6.10]{ABG-book})\,. Hence, letting $R\to\infty$ by Fatou's lemma from \cref{EErgonearPoisso1N},  it follows that
\begin{align*}
\hat{V}^v(x)&\,\geq\, \Exp_x^{v}\left[\int_0^{\uuptau_{r}} \left(c(X_t, v(X_t)) - \hat{\rho}_{v}\right) \D{t} +\hat{V}^{v}\left(X_{\uuptau_{r}}\right)\right]\nonumber\\
&\,\geq\, \Exp_x^{v}\left[\int_0^{\uuptau_{r}} \left(c(X_t, v(X_t)) - \hat{\rho}_{v}\right) \D{t}\right] +\inf_{\sB_r}\hat{V}^{v}\,.
\end{align*}Since $\hat{V}^{v}(0) =0$, letting $r\to 0$, we deduce that
\begin{align}\label{EErgonearPoisso1o}
\hat{V}^v(x)\,\geq\, \limsup_{r\downarrow 0}\Exp_x^{v}\left[\int_0^{\uuptau_{r}} \left(c(X_t, v(X_t)) - \hat{\rho}_{v}\right) \D{t} \right]\,.
\end{align}
From \cref{EErgonearPoisso1B} and \cref{EErgonearPoisso1o}, it is easy to see that $V^v - \hat{V}^v \leq 0$ in $\Rd$. On the other hand by \cref{EErgonearPoisso1A} and \cref{EErgonearPoisso1J} one has $\sL_{v}\left(V^v - \hat{V}^v\right)(x)\geq 0$ in $\Rd$. Hence, applying the strong maximum principle \cite[Theorem~9.6]{GilTru}, one has $V^v = \hat{V}^v$. This proves uniqueness.
\end{proof}
Now we prove the continuity of ergodic cost under near-monotonicity assumption on the running cost function. 
\begin{theorem}\label{ergodicnearmono1}
Suppose that Assumptions (A1)-(A4) hold. Let $\{v_n\}_n$ be a sequence of stable policies such that $v_n \to v$ in $\Usm$\, and $\{\eta_{v_n}\}_n$ tight. If $$\sup_{n}\sE_x(c, v_n) < \liminf_{\norm{x}\to\infty}\inf_{\zeta\in \Act} c(x,\zeta),$$ then we have the following
\begin{equation}\label{EErgonearOpt1A}
\inf_{\Rd}\sE_x(c, v_n) \to \inf_{\Rd}\sE_x(c, v)\quad \text{as}\,\,\,  n\to\infty\,.
\end{equation} 
\end{theorem} 
\begin{proof}
From Theorem~\ref{NearmonotPoisso}, we know that for each $n\in\NN$ there exists $(V^{v_n}, \rho_{v_n})\in \Sobl^{2,p}(\Rd)\times \RR$, \, $1< p < \infty$, with $V^{v_n}(0) = 0$ and $\inf_{\Rd} V^{v_n} > -\infty$, satisfying
\begin{equation}\label{EErgoContnuity1A}
\rho_{v_n} = \sL_{v_n}V^{v_n}(x) + c(x, v_n(x))\,,
\end{equation} where $\rho_{v_n} = \int_{\Rd}\int_{\Act} c(x, u)v_n(x)(\D{u})\eta_{v_n}(\D{x}) = \inf_{\Rd}\sE_x(c, v_n)$\,. Now from \cite[Lemma~4.4]{AB-10Uni}, since we impose tightness apriori, we deduce that $\eta_{v_n} \to \eta_{v}$ in total variation topology. Hence the associated densities $\varphi_{v_n}\to \varphi$ in $\Lp^1(\Rd)$ (see the proof of \cite[Lemma~3.2.5]{ABG-book}). Note that
\begin{align}
&\int_{\Rd}\int_{\Act} c(x,\zeta)v_{n}(x)(\D \zeta)\eta_{v_{n}}(\D x)\, -\int_{\Rd}\int_{\Act} c(x,\zeta) v(x)(\D \zeta)\eta_{v}(\D x) \nonumber \\
& =  \bigg(\int_{\Rd}\int_{\Act} c(x,\zeta)v_{n}(x)(\D \zeta)\varphi_{v_n}(x)\D x - \int_{\Rd}\int_{\Act} c(x,\zeta)v_{n}(x)(\D \zeta)\varphi(x)\D x \bigg) \nonumber \\
&\quad + \bigg(\int_{\Rd}\int_{\Act} c(x,\zeta)v_{n}(x)(\D \zeta)\varphi(x)\D x -\int_{\Rd}\int_{\Act} c(x,\zeta)v(x)(\D \zeta)\varphi(x)\D x \bigg)
\end{align}
Since $c$ is bounded, the first term of the right hand side converges to zero since $\varphi_{v_n}\to \varphi$ in $\Lp^1(\Rd)$ and the second term converges to zero by the convergence of $v_n\to v$ (see Definition~\ref{DefBorkarTopology1A})\,. Hence, it follows that $\int_{\Rd}\int_{\Act} c(x, u)v_n(x)(\D{u})\eta_{v_n}(\D{x}) \to \int_{\Rd}\int_{\Act} c(x, u)v(x)(\D{u})\eta_{v}(\D{x})$\,. Thus, in view of Theorem~\ref{NearmonotPoisso}, we obtain $\inf_{\Rd}\sE_x(c, v_n) \to \inf_{\Rd}\sE_x(c, v)$ as $n\to \infty$\,. This completes the proof.
\end{proof}
\begin{remark}
The tightness assumption is not superfluous. In view of \cite{AA13}, we know that the map $v\mapsto \inf_{\Rd}\sE_x(c, v)$ in general may not be continuous on $\Usm$ under near-monotone cost criterion (of the form, \cref{EErgonearOpt1A})\,. The reason is the following: for each $n\in\NN$ let $(V^{v_n}, \rho_{v_n})$ be the unique compatible solution pair (see, \cite[Definition ~1.1]{AA13}) of the equation \cref{EErgoContnuity1A}, if $(V^{v_n}, \rho_{v_n})$ converge to a solution pair $(\bar{V}, \bar{\rho})$ of the limiting equation of \cref{EErgoContnuity1A} as $n\to \infty$, the solution pair $(\bar{V}, \bar{\rho})$ may not necessarily be compatible (see, \cite{AA13}). One sufficient condition which ensure this continuity is the tightness of the space of corresponding invariant measures $\{\eta_{v_n}: n\in\NN\}$\,.  
\end{remark}
\subsubsection{Under Lyapunov stability}\label{Lyapunov stability}
In this section we study the continuity of ergodic cost criterion under Lyapunov stability assumption. We assume the following Lyapunov stability condition on the dynamics.
\begin{itemize}
\item[\hypertarget{A5}{{(A5)}}]
There exists a positive constant $\widehat{C}_0$, and a pair of inf-compact  functions $(\Lyap, h)\in \cC^{2}(\Rd)\times\cC(\Rd\times\Act)$ (i.e., the sub-level sets $\{\Lyap\leq k\} \,,\{h\leq k\}$ are compact or empty sets in $\Rd$\,, $\Rd\times\Act$ respectively for each $k\in\RR$) such that
\begin{equation}\label{Lyap1}
\sL_{\zeta}\Lyap(x) \leq \widehat{C}_{0} - h(x,\zeta)\quad \text{for all}\,\,\, (x,\zeta)\in \Rd\times \Act\,,
\end{equation} where  $h$ ($>0$) is locally Lipschitz continuous in its first argument uniformly with respect to the second and $\Lyap > 1$.
\end{itemize} 
A function $f\in\mathcal{O}(\Lyap)$ if $f\leq \widehat{C}_1 \Lyap$ for some positive constant $\widehat{C}_1$\, and $f \in \sorder(\Lyap)$ if $\displaystyle{\limsup_{|x|\to \infty} \frac{|f|}{\Lyap} = 0}$\,.
Now following \cite[Lemma~3.7.8]{ABG-book}, we want to prove that a certain equation admits a unique solution in some suitable function space. This uniqueness result is crucial to obtain continuity of the map $v\to \sE_x(c, v)$ on $\Usm$\,. 
\begin{theorem}\label{TErgoExis1}
Suppose that Assumptions (A1)-(A3) and (A5) hold. Then for each $v\in \Usm$ there exist a unique solution pair $(\widehat{V}^v, \widehat{\rho}^{v})\in \Sobl^{2,p}(\Rd)\cap \sorder(\Lyap)\times \RR$ for any $p >1$ satisfying
\begin{equation}\label{EErgoOpt1A}
\widehat{\rho}^{v} = \sL_{v}\widehat{V}^v(x) + c(x, v(x))\quad\text{with}\quad \widehat{V}^v(0) = 0\,.
\end{equation}
Furthermore, we have
\begin{itemize}
\item[(i)]$\widehat{\rho}^{v} = \sE_x(c, v)$
\item[(ii)] for all $x\in\Rd$, we have
\begin{equation}\label{EErgoOpt1C}
\widehat{V}^v(x) \,=\, \lim_{r\downarrow 0}\Exp_{x}^{v}\left[\int_{0}^{\uuptau_{r}} \left( c(X_t, v(X_t)) - \sE_x(c, v)\right)\D t\right]\,.
\end{equation}
\end{itemize} 
\end{theorem} 
\begin{proof}
Existence of a solution pair $(\widehat{V}^v, \widehat{\rho}^{v})\in \Sobl^{2,p}(\Rd)\cap \sorder(\Lyap)\times \RR$ for any $p >1$ satisfying (i) and (ii) follows from \cite[Lemma~3.7.8]{ABG-book}\,. Also, it is known that along a subsequence $\alpha\cJ_{\alpha}^{v}(0, c)\to \widehat{\rho}^{v}$ and $\cJ_{\alpha}^{v}(x, c) - \cJ_{\alpha}^{v}(0, c)\to \widehat{V}^v$ uniformly over compact subsets of $\Rd$ (see \cite[Lemma~3.7.8 (i)]{ABG-book})\,.

Next we show that the sub-sequential limits are unique\,. This indeed imply the uniqueness of the solutions. Let $(\bar{V}^v, \bar{\rho}^{v})\in \Sobl^{2,p}(\Rd)\cap \sorder(\Lyap)\times \RR$ for any $p >1$ be any other solution pair of \cref{EErgoOpt1A} with $\bar{V}^v(0) = 0$. Thus, by It$\hat{\rm o}$-Krylov formula, for $R>0$ we obtain
\begin{align}\label{ETErgoExis1A}
\Exp_{x}^{v}\left[\bar{V}^v(X_{T\wedge\uptau_{R}})\right] - \bar{V}^v(x) &= \Exp_{x}^{v}\left[\int_{0}^{T\wedge\uptau_{R}} \sL_{v} \bar{V}^v(X_s) \D s\right]\nonumber\\
& = \Exp_{x}^{v}\left[\int_{0}^{T\wedge\uptau_{R}} \left(\bar{\rho}^{v} - c(X_s, v(X_s))\right)\D s \right]\,.  
\end{align}
Note that 
\begin{equation*}
\int_{0}^{T\wedge\uptau_{R}} \left(\bar{\rho}^{v} - c(X_s, v(X_s))\right)\D s = \int_{0}^{T\wedge\uptau_{R}} \bar{\rho}^{v} - \int_{0}^{T\wedge\uptau_{R}}c(X_s, v(X_s))\D s
\end{equation*} Thus, letting $R\to \infty$ by monotone convergence theorem, we get
\begin{equation*}
\lim_{R\to\infty}\Exp_{x}^{v}\left[\int_{0}^{T\wedge\uptau_{R}} \left(\bar{\rho}^{v} - c(X_s, v(X_s))\right)\D s \right] = \Exp_{x}^{v}\left[\int_{0}^{T} \left(\bar{\rho}^{v} - c(X_s, v(X_s))\right)\D s \right]\,.
\end{equation*} Since $\bar{V}^v \in \sorder{(\Lyap)}$, in view of \cite[Lemma~3.7.2 (ii)]{ABG-book}, letting $R\to\infty$, we deduce that
\begin{align}\label{ETErgoExis1BA}
\Exp_{x}^{v}\left[\bar{V}^v(X_{T})\right] - \bar{V}^v(x) = \Exp_{x}^{v}\left[\int_{0}^{T} \left(\bar{\rho}^{v} - c(X_s, v(X_s))\right)\D s \right]\,.  
\end{align} Also, from \cite[Lemma~3.7.2 (ii)]{ABG-book}, we have
\begin{equation*}
\lim_{T\to\infty}\frac{\Exp_{x}^{v}\left[\bar{V}^v(X_{T})\right]}{T} = 0\,. 
\end{equation*}
Hence, dividing both sides of \cref{ETErgoExis1BA} by $T$ and letting $T\to\infty$, we obtain
\begin{align*}
\bar{\rho}^{v} = \limsup_{T\to \infty}\frac{1}{T}\Exp_{x}^{v}\left[\int_{0}^{T} \left(c(X_s, v(X_s))\right)\D s \right]\,.
\end{align*}This implies that $\bar{\rho}^{v} = \widehat{\rho}^{v}$\,. Again, applying It$\hat{\rm o}$-Krylov formula and using \cref{EErgoOpt1A}, we have
\begin{align}\label{ETErgoExis1B}
\bar{V}^v(x)\,=\, \Exp_x^{v}\left[\int_0^{\uuptau_{r}\wedge \uptau_{R}} \left(c(X_t, v(X_t)) - \bar{\rho}^{v}\right) \D{t} + \bar{V}^{v}\left(X_{\uuptau_{r}\wedge \uptau_{R}}\right)\right]\,.
\end{align} Also, from \cref{Lyap1}, by It$\hat{\rm o}$-Krylov formula it follows that
\begin{align*}
\Exp_x^{v}\left[\Lyap\left(X_{\uuptau_{r}\wedge \uptau_{R}}\right)\right] - \Lyap(x)\,=\, \Exp_x^{v}\left[\int_0^{\uuptau_{r}\wedge \uptau_{R}} \sL_{v}\Lyap(X_t) \D{t}\right] \leq \Exp_x^{v}\left[\int_0^{\uuptau_{r}\wedge \uptau_{R}} \left(\widehat{C}_0 - h(X_t, v(X_t))\right) \D{t}\right]\,.
\end{align*}
This gives us the following (since $h(x,\zeta) > 0$)
\begin{equation*}
\Exp_x^{v}\left[\Lyap\left(X_{\uptau_{R}}\right)\Ind_{\{\uuptau_{r}\geq \uptau_{R}\}}\right]\leq \widehat{C}_0 \Exp_x^{v}\left[\uuptau_{r}\right] + \Lyap(x)\quad \text{for all} \,\,\, r <|x|<R\,.
\end{equation*} Now, it is easy to see that
\begin{equation*}
-\sup_{\partial{\sB_R}}\frac{|\hat{V}^{v}|}{\Lyap} \left(\widehat{C}_0 \Exp_x^{v}\left[\uuptau_{r}\right] + \Lyap(x)\right)\leq \Exp_x^{v}\left[\hat{V}^{v}\left(X_{\uptau_{R}}\right)\Ind_{\{\uuptau_{r}\geq \uptau_{R}\}}\right] \leq \sup_{\partial{\sB_R}}\frac{|\hat{V}^{v}|}{\Lyap} \left(\widehat{C}_0 \Exp_x^{v}\left[\uuptau_{r}\right] + \Lyap(x)\right)\,.
\end{equation*}
Since $\bar{V}^v \in \sorder(\Lyap)$, from the above estimate, we get 
\begin{equation*}
\liminf_{R\to\infty}\Exp_x^{v}\left[\hat{V}^{v}\left(X_{\uptau_{R}}\right)\Ind_{\{\uuptau_{r}\geq \uptau_{R}\}}\right] = 0\,.
\end{equation*}Thus, letting $R\to\infty$ by Fatou's lemma from \cref{ETErgoExis1B}, it follows that
\begin{align*}
\bar{V}^v(x)&\,\geq\, \Exp_x^{v}\left[\int_0^{\uuptau_{r}} \left(c(X_t, v(X_t)) - \bar{\rho}^{v}\right) \D{t} +\bar{V}^{v}\left(X_{\uuptau_{r}}\right)\right]\nonumber\\
&\,\geq\, \Exp_x^{v}\left[\int_0^{\uuptau_{r}} \left(c(X_t, v(X_t)) - \bar{\rho}^{v}\right) \D{t}\right] +\inf_{\sB_r}\bar{V}^{v}\,.
\end{align*}Since $\bar{V}^{v}(0) =0$, letting $r\to 0$, we deduce that
\begin{align}\label{ETErgoExis1C}
\bar{V}^v(x)\,\geq\, \limsup_{r\downarrow 0}\Exp_x^{v}\left[\int_0^{\uuptau_{r}} \left(c(X_t, v(X_t)) - \bar{\rho}^{v}\right) \D{t} \right]\,.
\end{align}
Since $\widehat{\rho}^v = \bar{\rho}^v$, from \cref{EErgoOpt1C} and \cref{ETErgoExis1C}, it follows that $\widehat{V}^v - \bar{V}^v \leq 0$ in $\Rd$. Also, since $(\widehat{V}^v, \widehat{\rho}^v)$ and $(\bar{V}^v, \bar{\rho}^v)$ are two solution pairs of \cref{EErgoOpt1A}, we have $\sL_{v}\left(\widehat{V}^v - \bar{V}^v\right)(x) = 0$ in $\Rd$. Hence, by strong maximum principle \cite[Theorem~9.6]{GilTru}, one has $\widehat{V}^v = \bar{V}^v$. This proves uniqueness
\end{proof}
Next we prove that the map $v\to \inf_{\Rd}\sE_x(c, v)$ is continuous on $\Usm$ under the Borkar topology\,.
\begin{theorem}\label{ergodicLyap1}
Suppose that Assumptions (A1)-(A3) and (A5) hold. Let $\{v_n\}_n$ be a sequence of policies in $\Usm$ such that $v_n \to v$ in $\Usm$\,. Then we have 
\begin{equation}\label{EErgoLyap1A}
\inf_{\Rd}\sE_x(c, v_n) \to \inf_{\Rd}\sE_x(c, v)\quad \text{as}\,\,\, n\to\infty\,.
\end{equation} 
\end{theorem}
\begin{proof}
From Theorem~\ref{TErgoExis1}, we know that for each $n\in \NN$ there exists unique solution pair $(\widehat{V}^{v_n}, \widehat{\rho}^{v_n})\in \Sobl^{2,p}(\Rd)\cap \sorder(\Lyap)\times \RR$ for any $p >1$ satisfying
\begin{equation}\label{EErgoLyap1B}
\widehat{\rho}^{v_n} = \sL_{v_n}\widehat{V}^{v_n}(x) + c(x, v_n(x))\quad\text{with}\quad \widehat{V}^{v_n}(0) = 0\,,
\end{equation}
where
\begin{itemize}
\item[(i)]$\widehat{\rho}^{v_n} = \sE_x(c, v_n)$
\item[(ii)] for all $x\in\Rd$, we have
\begin{equation*}
\widehat{V}^{v_n}(x) \,=\, \lim_{r\downarrow 0}\Exp_{x}^{v_n}\left[\int_{0}^{\uuptau_{r}} \left( c(X_t, v_n(X_t)) - \sE_x(c, v_n)\right)\D t\right]\,.
\end{equation*}
\end{itemize} 
In view of \cref{Lyap1}, it is easy to see that, each $v\in\Usm$ is stable and $\inf_{v\in\Usm}\eta_v(\sB_R) > 0$ for any $R>0$ (see, \cite[Lemma~3.3.4]{ABG-book} and \cite[Lemma~3.2.4(b)]{ABG-book}). Thus, from \cite[Theorem~3.7.4]{ABG-book}, it follows that
\begin{equation}\label{EErgoLyap1C}
\norm{\cJ_{\alpha}^{v_n}(\cdot, c) - \cJ_{\alpha}^{v_n}(0, c)}_{\Sob^{2,p}(\sB_R)}\leq \frac{\widehat{C}_{2}(R, p)}{\eta_{v_n}(\sB_{2R})} \left(\frac{\widehat{\rho}^{v_n}}{\eta_{v_n}(\sB_{2R})} + \sup_{\sB_{4R}\times \Act}c(x,\zeta) \right)\,,
\end{equation} where the positive constant $\widehat{C}_{2}(R, p)$ depends only on $R$ and $p$\,. Since the running cost is bounded we have $\|c\|_{\infty} \leq M$ for some positive constant $M$. Thus, we have $\widehat{\rho}^{v_n} \leq M$. Hence from \cref{EErgoLyap1C}, we deduce that
\begin{equation*}
\norm{\cJ_{\alpha}^{v_n}(\cdot, c) - \cJ_{\alpha}^{v_n}(0, c)}_{\Sob^{2,p}(\sB_R)}\leq \frac{M\widehat{C}_{2}(R, p)}{\inf_{n}\eta_{v_n}(\sB_{2R})} \left(\frac{1}{\inf_{n}\eta_{v_n}(\sB_{2R})} + 1 \right)\,.
\end{equation*} This implies that 
\begin{equation}\label{EErgoLyap1D}
\norm{\widehat{V}^{v_n}}_{\Sob^{2,p}(\sB_R)} \leq \widehat{C}_{3}(R, p)\,,
\end{equation}where $\widehat{C}_{3}(R, p)$ is a positive constant which depends only on $R$ and $p$\,. Hence, by a standard diagonalization argument and Banach Alaoglu theorem (see, \cref{ET1.1C}), one can extract a subsequence $\{\widehat{V}^{v_{n_k}}\}$ such that for some $\widehat{V}^*\in \Sobl^{2,p}(\Rd)$ we have
\begin{equation}\label{EErgoLyap1E}
\begin{cases}
\widehat{V}^{v_{n_k}}\to & \widehat{V}^*\quad \text{in}\quad \Sobl^{2,p}(\Rd)\quad\text{(weakly)}\\
\widehat{V}^{v_{n_k}}\to & \widehat{V}^*\quad \text{in}\quad \cC^{1, \beta}_{loc}(\Rd) \quad\text{(strongly)}\,.
\end{cases}       
\end{equation}Also, since $\widehat{\rho}^{v_n} \leq M$, along a further subsequence $\widehat{\rho}^{v_{n_k}}\to \widehat{\rho}^{*}$ (without loss of generality denoting by the same sequence). Now, by similar argument as in Theorem~\ref{T1.1}, multiplying by test function on the both sides of \cref{EErgoLyap1B} and letting $k\to\infty$, we deduce that $(\widehat{V}^*, \widehat{\rho}^{*})\in \Sobl^{2,p}(\Rd)\times \RR$ satisfies
\begin{equation}\label{EErgoLyap1F}
\widehat{\rho}^{*} = \sL_{v}\widehat{V}^{*}(x) + c(x, v(x))\,.
\end{equation}Since $\widehat{V}^{v_n}(0) = 0$ for each $n$, we get $\widehat{V}^*(0) = 0$
  
Next we want to show that $\widehat{V}^{*}\in \sorder{(\Lyap)}$. Following the proof of \cite[Lemma~3.7.8]{ABG-book} (see, eq.(3.7.47) or eq.(3.7.50)), it is easy to see that 
\begin{equation*}
\widehat{V}^{v_n}(x) \,\leq\, \Exp_{x}^{v_n}\left[\int_{0}^{\uuptau_{r}} \left( c(X_t, v_n(X_t)) - \sE_x(c, v_n)\right)\D t + \widehat{V}^{v_n}(X_{\uuptau_{r}})\right]\,.
\end{equation*}
This gives us the following estimate
\begin{equation}\label{EErgoLyap1G}
|\widehat{V}^{v_n}(x)| \,\leq\, M\sup_{n}\Exp_{x}^{v_n}\left[\int_{0}^{\uuptau_{r}} \left( c(X_t, v_n(X_t)) + 1\right)\D t + \sup_{\sB_r}|\widehat{V}^{v_n}|\right]\,.
\end{equation} We know that, for $d < p < \infty$, the space $\Sob^{2,p}(\sB_R)$ is compactly embedded in $\cC^{1, \beta}(\bar{\sB}_R)$\,, where $\beta < 1 - \frac{d}{p}$ (see \cite[Theorem~A.2.15 (2b)]{ABG-book}). Thus, from \cref{EErgoLyap1D}, we obtain $\displaystyle{\sup_{n}\sup_{\sB_r}|\widehat{V}^{v_n}| < \widehat{M}}$ for some positive constant $\widehat{M}$\,. Therefore, in view of \cite[Lemma~3.7.2 (i)]{ABG-book}, form \cref{EErgoLyap1G}, we deduce that $\widehat{V}^{*}\in \sorder{(\Lyap)}$\,. Since the pair $(\widehat{V}^*, \widehat{\rho}^{*})\in \Sobl^{2,p}(\Rd)\cap \sorder{(\Lyap)} \times \RR$ satisfies \cref{EErgoLyap1F}, by uniqueness of solution of \cref{EErgoLyap1F} (see, Theorem~\ref{TErgoExis1}) it follows that $(\widehat{V}^*, \widehat{\rho}^{*})\equiv (\widehat{V}^v, \widehat{\rho}^{v})$\,. This completes the proof of the theorem\,.
\end{proof}

\section{Denseness of Finite Action/Piecewise Constant Stationary Policies}\label{DensePol}
\subsection{Denseness of Policies with Finite Actions}
Let $d_{\Act}$ be the metric on the action space $\Act$\,. Since $\Act$ is compact, we have $\Act$ is totally bounded. Thus, one can find a sequence of finite grids $\{\{\zeta_{n,i}\}_{i=1}^{k_n}\}_{n\geq 1}$ such that $$\min_{i= 1,2,\dots , k_n}d(\zeta, \zeta_{n,i}) < \frac{1}{n}\quad\text{for all}\,\,\, \zeta\in \Act\,.$$

Let $\Lambda_{n} \df \{\zeta_{n,1}, \zeta_{n,2}, \dots ,\zeta_{n,k_n}\}$ and define a function $Q_n: \Act\to \Lambda_{n}$ by 
\begin{equation*}
Q_n(\zeta) = \argmin_{\zeta_{n,i}\in \Lambda_n} d(\zeta, \zeta_{n,i})\,,
\end{equation*}
where ties are broken so that $Q_n$ is measurable. The function $Q_n$ is often known as nearest neighborhood quantizer (see, \cite{SYL17}). 

For each $n$ the function $Q_n$ induces a partition $\{\Act_{n,i}\}_{i=1}^{k_n}$ of the action space $\Act$ given by
\begin{equation*}
\Act_{n,i} = \{\zeta\in \Act : Q_n(\zeta) = \zeta_{n,i}\}\,.
\end{equation*} By triangle inequality, it follows that $\text{diam}(\Act_{n,i})\df \sup_{\zeta_1, \zeta_2 \in \Act_{n,i}} d_{\Act}(\zeta_1, \zeta_2) < \frac{2}{n}$\,. Now, for each $v\in\Usm$ define a sequence of policies with finite actions as follows: 
\begin{equation}\label{DenseStra1}
v_{n}(\zeta_{n,i}|x) = Q_n v(\zeta_{n,i}|x) = v(\Act_{n,i}|x)\,.
\end{equation}
In the next lemma we prove that the space of stationary policies with finite actions are dense in $\Usm$ with respect to the \emph{Borkar topology} (see, Definition~\ref{DefBorkarTopology1A}) \,.
\begin{lemma}\label{DenseBorkarTopo}
For each $v\in\Usm$ there exists a sequence of policies $\{v_n\}_n$ (defined as in \cref{DenseStra1}) with finite actions, satisfying
\begin{equation}\label{DenseStra2}
\lim_{n\to\infty}\int_{\Rd}f(x)\int_{\Act}g(x,\zeta)v_{n}(x)(\D \zeta)\D x = \int_{\Rd}f(x)\int_{\Act}g(x,\zeta)v(x)(\D \zeta)\D x
\end{equation}
for all $f\in L^1(\Rd)\cap L^2(\Rd)$ and $g\in \cC_b(\Rd\times \Act)$
\end{lemma}
\begin{proof} Let $f\in L^1(\Rd)\cap L^2(\Rd)$ and $g\in \cC_b(\Rd\times \Act)$. Then from the construction of the sequence $\{v_n\}_n$, it is easy to see that
\begin{align*}
|\int_{\Rd}f(x)\int_{\Act}g(x,u)v_{n}(x)(\D u)\D x  & - \int_{\Rd}f(x)\int_{\Act}g(x,\zeta)v(x)(\D \zeta)\D x |\\
&\leq  \int_{\Rd}|f(x)|\sum_{i=1}^{k_n}\int_{\Act_{n,i}}|g(x,\zeta_{n,i}) - g(x,\zeta)|v(x)(\D \zeta)\D x\,. 
\end{align*}
Since $g\in\cC_b(\Rd\times \Act)$ and $\text{diam}(\Act_{n,i}) < \frac{2}{n}$, it follows that
\begin{equation*}
|f(x)|\sum_{i=1}^{k_n}\int_{\Act_{n,i}}|g(x,\zeta_{n,i}) - g(x,\zeta)|v(x)(\D \zeta)\rightarrow 0\quad\text{for all}\,\,\, x\in\Rd\,.
\end{equation*} As we know that $g$ is bounded, for some positive constant $M_1$ we have $|g| \leq M_1$. Thus, we deduce that
\begin{equation*}
|f(x)|\sum_{i=1}^{k_n}\int_{\Act_{n,i}}|g(x,\zeta_{n,i}) - g(x,\zeta)|v(x)(\D \zeta)\leq 2M_1 |f(x)| \quad\text{for all}\,\,\, x\in\Rd\,.
\end{equation*}Since $f\in L^1(\Rd)\cap L^2(\Rd)$, by dominated convergence theorem, we obtain
\begin{equation*}
\lim_{n\to\infty}\int_{\Rd}f(x)\int_{\Act}g(x,u)v_{n}(x)(\D u)\D x = \int_{\Rd}f(x)\int_{\Act}g(x,\zeta)v(x)(\D \zeta)\D x\,.
\end{equation*}This completes the proof of the lemma.
\end{proof}
\subsection{Denseness of Piecewise Constant Policies}
Let $d_{\mathscr{P}}$ be the Prokhorov metric on $\pV$\,. Since $(\Act, d_{\Act})$ is separable (being a compact metric space) thus convergence in $(\pV, d_{\mathscr{P}})$ is equivalent to weak convergence of probability measures.
\begin{theorem}\label{TDPCP}
For each $v\in \Usm$ there exists a sequence of piecewise constant policies $\{v_m\}_{m}$ in $\Usm$ such that
\begin{equation}\label{BorkarTopology2}
\lim_{m\to\infty}\int_{\Rd}f(x)\int_{\Act}g(x,\zeta)v_{m}(x)(\D \zeta)\D x = \int_{\Rd}f(x)\int_{\Act}g(x,\zeta)v(x)(\D \zeta)\D x
\end{equation}
for all $f\in L^1(\Rd)\cap L^2(\Rd)$ and $g\in \cC_b(\Rd\times \Act)$ 
\end{theorem}
\begin{proof}
Let $\sB_{0} = \emptyset$ and define $D_{n} = \sB_{n}\setminus \sB_{n-1}$ for $n\in \NN$\,. Thus it is easy to see that $\Rd = \cup_{n=1}^{\infty} D_{n}$. Since  each $v\in \Usm$ is a measurable map $v: \Rd \to \pV$, it follows that $\hat{v}_{n}\,\df\, v\arrowvert_{D_n} : D_n \to \pV$ is a measurable map. Hence, by Lusin's theorem (see \cite[Theorem~7.5.2]{D02-book}), for any $\epsilon_n > 0$ there exists a compact set $K_{n}^{\epsilon_n}\subset D_n$ and a continuous function $\hat{v}_{n}^{\epsilon_n} : K_{n}^{\epsilon_n}\to \pV$ such that (the Lebesgue measure of the set $D_n\setminus K_{n}^{\epsilon_n}$) $\arrowvert(D_n\setminus K_{n}^{\epsilon_n}) \arrowvert < \epsilon_n$ and $\hat{v}_{n} \equiv \hat{v}_{n}^{\epsilon_n}$ on $K_{n}^{\epsilon_n}$\,. Again, Tietze's extension theorem (see \cite[Theorem~4.1]{DG51}) there exists a continuous function $\tilde{v}_{n}^{\epsilon_n}: D_n \to \pV$ such that $ \tilde{v}_{n}^{\epsilon_n}\equiv \hat{v}_{n}^{\epsilon_n}$ on $K_{n}^{\epsilon_n}$\,. 
\begin{itemize}
\item[\textbf{Step1}]
Therefore for any $\hat{f}\in \Lp^1(\Rd)\cap \Lp^2(\Rd)$ and $\hat{g}\in \cC(\Act)$, we have
\begin{align}\label{EBT1}
&\arrowvert \int_{D_n}\hat{f}(x) \int_{\Act} \hat{g}(\zeta)\hat{v}_{n}(x)(\D \zeta)\D x - \int_{D_n}\hat{f}(x)\int_{\Act} \hat{g}(\zeta)\tilde{v}_{n}^{\epsilon_n}(x)(\D \zeta)\D x \arrowvert \nonumber\\
&\leq \arrowvert \int_{D_n\setminus K_{n}^{\epsilon_n}} \hat{f}(x) \int_{\Act} \hat{g}(\zeta)\hat{v}_{n}(x)(\D \zeta)\D x - \int_{D_n\setminus K_{n}^{\epsilon_n}} \hat{f}(x) \int_{\Act} \hat{g}(\zeta)\tilde{v}_{n}^{\epsilon_n}(x)(\D \zeta)\D x \arrowvert \nonumber\\
&\leq \arrowvert \int_{D_n\setminus K_{n}^{\epsilon_n}} \hat{f}(x) \int_{\Act} \hat{g}(\zeta)\hat{v}_{n}(x)(\D \zeta)\D x \arrowvert + \arrowvert\int_{D_n\setminus K_{n}^{\epsilon}}\hat{f}(x)\int_{\Act} \hat{g}(\zeta)\tilde{v}_{n}^{\epsilon_n}(x)(\D \zeta)\D x \arrowvert \nonumber\\
&\leq \|\hat{g}\|_{\infty}\int_{D_n\setminus K_{n}^{\epsilon_n}}|\hat{f}(x)|\D x  + \|\hat{g}\|_{\infty}\int_{D_n\setminus K_{n}^{\epsilon_n}}|\hat{f}(x)|\D x \nonumber\\
&\leq 2\|\hat{g}\|_{\infty}\|\hat{f}\|_{\Lp^2(\Rd)} \sqrt{|(D_n\setminus K_{n}^{\epsilon_n})|} \leq 2\sqrt{\epsilon_n} \|\hat{g}\|_{\infty}\|\hat{f}\|_{\Lp^2(\Rd)}\quad\text{(by H\"older's inequality)}\,.
\end{align}
Now, since $(\pV, d_{\mathscr{P}})$ is compact, for each $m\in \NN$ there exists a finite set $\widehat{\Lambda}_{m} = \{\mu_{m,1}, \mu_{m,2}, \dots , \mu_{m, k_m}\}$ such that $$\inf_{\mu_{m, i}\in \widehat{\Lambda}_{m}}d_{\mathscr{P}}(\mu, \mu_{m, i}) < \frac{1}{m}\quad \text{for any}\quad \mu\in \pV\,.$$ Let $\widehat{Q}_{m}: \pV \to \widehat{\Lambda}_{m}$ be defined as 
\begin{equation*}
\widehat{Q}_{m} (\mu) = \argmin_{\mu_{m,i}\in\widehat{\Lambda}_{m}} d_{\mathscr{P}}(\mu, \mu_{m,i})\,.
\end{equation*} Ties are broken so that $\widehat{Q}_{m}$ is a measurable map. Hence, it induces a partition $\{\widehat{U}_{m,i}\}_{i=1}^{k_m}$ of the space $\pV$ which is given by
\begin{equation*}
\widehat{U}_{m,i} = \{\mu\in\pV: \widehat{Q}_{m}(\mu) = \mu_{m,i}\}\,.
\end{equation*}By triangle inequality it is easy to see that 
\begin{equation*}
\diam(\widehat{U}_{m,i}) \df \sup_{\mu_1, \mu_2}d_{\mathscr{P}}(\mu_1, \mu_2) < \frac{2}{m}\,. 
\end{equation*} Now, for $v\in\Usm$ define $D_{n,i}^m = (\tilde{v}_{n}^{\epsilon_n})^{-1}(\widehat{U}_{m,i})$. This implies that $D_{n} = \cup_{i=1}^{k_m} D_{n,i}^m$\,. Define
\begin{equation*}
\hat{v}_{n,m}^{\epsilon_n}(x) \df \sum_{i=1}^{k_m} \mu_{m,i}\Ind_{\{D_{n,i}^m\}}(x)\quad\text{for all}\quad x\in D_n\,\,\,\text{and}\,\,\,m\in \NN\,.
\end{equation*} Therefore, we deduce that
\begin{align}\label{EBT2}
&\arrowvert \int_{D_n} \hat{f}(x)\int_{\Act} \hat{g}(\zeta)\tilde{v}_{n}^{\epsilon_n}(x)(\D \zeta)\D x - \int_{D_n} \hat{f}(x) \int_{\Act} \hat{g}(\zeta)\hat{v}_{n,m}^{\epsilon_n}(x)(\D \zeta)\D x \arrowvert \nonumber\\
&\leq \sum_{i=1}^{k_m} \arrowvert \int_{D_{n,i}^{m}} \hat{f}(x) \int_{\Act} \hat{g}(\zeta)\tilde{v}_{n}^{\epsilon_n}(x)(\D \zeta)\D x - \int_{D_{n,i}^{m}} \hat{f}(x) \int_{\Act} \hat{g}(\zeta)\mu_{m,i}(\D \zeta)\D x \arrowvert \nonumber\\
&\leq \sum_{i=1}^{k_m} \int_{D_{n,i}^{m}} |\hat{f}(x)| \arrowvert \int_{\Act} \hat{g}(\zeta)\tilde{v}_{n}^{\epsilon_n}(x)(\D \zeta) - \int_{\Act} \hat{g}(\zeta)\mu_{m,i}(\D \zeta)\arrowvert \D x  \nonumber\\
&\leq \|\hat{f}\|_{\Lp^1(\Rd)}\epsilon_n \quad\text{(for large enough $m$)}\,.
\end{align} If we choose $\epsilon_n = \min\{\frac{\epsilon_n\|\hat{f}\|_{\Lp^2(\Rd)}\|\hat{g}\|_{\infty}}{4}, \frac{\epsilon_n\|\hat{f}\|_{\Lp^1(\Rd)}}{2}\}$, combining \cref{EBT1}, \cref{EBT2}, there exists $\bar{M}_0 >0$ (depending on $\hat{f}, \hat{g}$ and $\epsilon_n$) such that
\begin{align}\label{EBT3}
\arrowvert \int_{D_n}\hat{f}(x) \int_{\Act} \hat{g}(\zeta)\hat{v}_{n}(x)(\D \zeta)\D x - \int_{D_n}\hat{f}(x)\int_{\Act} \hat{g}(\zeta)\hat{v}_{n,m}^{\epsilon_n}(x)(\D \zeta)\D x \arrowvert \leq \epsilon_n\,,
\end{align} for all $m\geq \bar{M}_0$\,.
\item[\textbf{Step2}]
Let $\epsilon > 0$ be a small number. Now define 
\begin{equation}\label{EBT4}
\bar{v}_{m}^{\epsilon} \df \sum_{n=1}^{\infty} \hat{v}_{n,m}^{\epsilon_n} \quad \text{for}\,\,\, m\in \NN\,. 
\end{equation} Since $\hat{f}\in\Lp^1(\Rd)$ there exists $N_0 \in \NN$ such that $\int_{\sB_{N_0}^c}|\hat{f}(x)|\D x < \frac{\epsilon}{4\|\hat{g}\|_{\infty}}$
\begin{align*}
&\arrowvert \int_{\Rd}\hat{f}(x)\int_{\Act} \hat{g}(\zeta)v(x)(\D \zeta)\D x - \int_{\Rd}\hat{f}(x)\int_{\Act} \hat{g}(\zeta)\bar{v}_{m}^{\epsilon}(x)(\D \zeta)\D x \arrowvert \nonumber\\
&\leq \arrowvert \int_{\sB_{N_0}^c} \hat{f}(x)\int_{\Act} \hat{g}(\zeta)(v(x) - \bar{v}_{m}^{\epsilon}(x))(\D \zeta)\D x \arrowvert + \arrowvert \int_{\sB_{N_0}} \hat{f}(x)\int_{\Act} \hat{g}(\zeta)(v(x) - \bar{v}_{m}^{\epsilon}(x))(\D \zeta)\D x \arrowvert \nonumber\\
&\leq \frac{\epsilon}{2} + \arrowvert \int_{\sB_{N_0}} \hat{f}(x)\int_{\Act} \hat{g}(\zeta)(v(x) - \bar{v}_{m}^{\epsilon}(x))(\D \zeta)\D x \arrowvert
\end{align*}
Now, choose $\epsilon_{i} > 0$ for $i= 1,\dots , N_0$ such that $\sum_{i = 1}^{N_0} \epsilon_i < \frac{\epsilon}{2}$. Thus,  in view of \cref{EBT3} there exists $M_i >0$ such that for each $i = 1, \dots , N_0$
\begin{equation*}
\arrowvert \int_{D_i}\hat{f}(x) \int_{\Act} \hat{g}(\zeta)\hat{v}_{i}(x)(\D \zeta)\D x - \int_{D_i}\hat{f}(x)\int_{\Act} \hat{g}(\zeta)\hat{v}_{i,m}^{\epsilon_i}(x)(\D \zeta)\D x \arrowvert \leq \epsilon_i\,,
\end{equation*} for all $m\geq M_i$\,. Hence, for $m\geq \max\{M_i,\, i = 1,\dots ,N_0\}$, we get    
\begin{align}\label{EBT5}
\arrowvert \int_{\sB_{N_0}} \hat{f}(x)\int_{\Act} \hat{g}(\zeta)(v(x) & - \bar{v}_{m}^{\epsilon}(x))(\D \zeta)\D x \arrowvert \nonumber \\
& \leq \sum_{i=1}^{N_0} |\int_{D_{i}}\hat{f}(x)\int_{\Act} \hat{g}(\zeta)(\hat{v}_{i}(x) - \hat{v}_{i,m}^{\epsilon_i}(x))(\D \zeta)\D x \arrowvert \leq \sum_{i = 1}^{N_0} \epsilon_i < \frac{\epsilon}{2}\,.
\end{align}
Therefore, for each $\epsilon >0$ we deduce that there exists a positive constant $\hat{M}_0$ (= $\max\{M_i,\, i = 1,\dots ,N_0\}$) such that for $m\geq \hat{M}_0$ (where $\hat{M}_0$ depends on $\hat{f}, \hat{g}, \epsilon$)
\begin{equation}\label{EBT6}
\arrowvert \int_{\Rd}\hat{f}(x)\int_{\Act} \hat{g}(\zeta)v(x)(\D \zeta)\D x - \int_{\Rd}\hat{f}(x)\int_{\Act} \hat{g}(\zeta)\bar{v}_{m}^{\epsilon}(x)(\D \zeta)\D x \arrowvert \leq \epsilon\,.
\end{equation}

\item[\textbf{Step3}]
Let $\{\hat{f}_k\}_{k\in\NN}$ and $\{h_{j}\}_{j\in\NN}$ be countable dense set in $\Lp^1(\Rd)$ and $\cC(\Act)$ respectively\,. Thus \cref{EBT6} holds true for each $\hat{f}_k$ and $h_j$\,.
 
Let $f\in\Lp^{1}(\Rd)\cap \Lp^{2}(\Rd)$ and $g\in\cC_{b}(\Rd\times \Act)$\,. Since $f\in \Lp^{1}(\Rd)$ for $\epsilon > 0$ there exists $N_{1}\in \NN$ such that $\int_{\sB_{N_1}^c} |f(x)|\D x \leq \frac{\epsilon}{4\|g\|_{\infty}}$\,. This implies
\begin{align}\label{EBT7}
&\arrowvert \int_{\Rd}f(x)\int_{\Act} g(x,\zeta)v(x)(\D \zeta)\D x - \int_{\Rd}f(x)\int_{\Act} g(x,\zeta)\bar{v}_{m}^{\epsilon}(x)(\D \zeta)\D x \arrowvert \nonumber\\
&\leq \arrowvert \int_{\sB_{N_1}^c} f(x)\int_{\Act} g(x,\zeta)(v(x) - \bar{v}_{m}^{\epsilon}(x))(\D \zeta)\D x \arrowvert + \arrowvert \int_{\sB_{N_1}} f(x)\int_{\Act} g(x,\zeta)(v(x) - \bar{v}_{m}^{\epsilon}(x))(\D \zeta)\D x \arrowvert \nonumber\\
&\leq \frac{\epsilon}{2} + \arrowvert \int_{\sB_{N_1}} f(x)\int_{\Act} g(x, \zeta)(v(x) - \bar{v}_{m}^{\epsilon}(x))(\D \zeta)\D x \arrowvert\,.
\end{align}
It is well known that in $\cC_b(\bar{\sB}_{N_1}\times \Act)$ the functions of the form   $\{\sum_{i}^{m} r_{i}(x)p_i(\zeta)\}_{m\in\NN}$ forms an algebra which contains constants, where $r_i\in \cC(\bar{\sB}_{N_1})$ and $p_i \in \cC(\Act)$\,. Thus by Stone-Weierstrass theorem there exists $\hat{m}$ (large enough) such that 
\begin{equation}\label{EBT8}
\sup_{\sB_{N_1}\times \Act} |g(x,\zeta) - \sum_{i}^{\hat{m}} r_{i}(x)p_i(\zeta)| \leq \frac{\epsilon}{24\|f\|_{\Lp^1(\Rd)}}\,.
\end{equation}
Since $p_i \in \cC(\Act)$ we can find $h_{j(i)} \in \cC(\Act)$ such that
\begin{equation}\label{EBT9}
\sup_{\zeta\in \Act} |p_i(\zeta) - h_{j(i)}(\zeta)| \leq \frac{\epsilon}{24\|f\|_{\Lp^1(\Rd)}\|r_{i}\|_{\infty}}\,.
\end{equation}
Also, since $fr_i\in \Lp^1(\Rd)$ there exists $\hat{f}_{k(i)}$ such that
\begin{equation}\label{EBT9A}
\int_{\sB_{N_1}} |f(x)r_i(x) - \hat{f}_{k(i)}(x)|\D x \leq \frac{\epsilon}{24\|f\|_{\Lp^1(\Rd)}\|h_{i}\|_{\infty}}\,.
\end{equation} Now, using \cref{EBT8}, \cref{EBT9}, \cref{EBT9A} we have the following
\begin{align}\label{EBT10}
&\arrowvert \int_{\sB_{N_1}} f(x)\int_{\Act} g(x, \zeta)v(x) - \int_{\sB_{N_1}} f(x)\int_{\Act} g(x, \zeta)\bar{v}_{m}^{\epsilon}(x))(\D \zeta)\D x \arrowvert\nonumber\\
\leq & \arrowvert \int_{\sB_{N_1}} f(x)\int_{\Act} g(x, \zeta)v(x) (\D \zeta)\D x - \int_{\sB_{N_1}} f(x)\int_{\Act} \sum_{i}^{\hat{m}} r_{i}(x)p_i(\zeta) v(x)(\D \zeta)\D x \arrowvert\nonumber\\
& + \sum_{i=1}^{\hat{m}}\arrowvert \int_{\sB_{N_1}} f(x)\int_{\Act} r_{i}(x)p_i(\zeta) v(x)(\D \zeta)\D x - \int_{\sB_{N_1}} f(x)\int_{\Act} r_{i}(x)h_{j(i)}(\zeta)v(x) (\D \zeta)\D x\arrowvert\nonumber\\
& + \sum_{i=1}^{\hat{m}}\arrowvert \int_{\sB_{N_1}} f(x)\int_{\Act} r_{i}(x)h_{j(i)}(\zeta)v(x) (\D \zeta)\D x - \int_{\sB_{N_1}} \hat{f}_{k(i)}(x)\int_{\Act} h_{j(i)}(\zeta)v(x) (\D \zeta)\D x\arrowvert\nonumber\\
& + \sum_{i=1}^{\hat{m}}\arrowvert  \int_{\sB_{N_1}} \hat{f}_{k(i)}(x)\int_{\Act} h_{j(i)}(\zeta)v(x) (\D \zeta)\D x - \int_{\sB_{N_1}} \hat{f}_{k(i)}(x)\int_{\Act} h_{j(i)}(\zeta)\bar{v}_{m}^{\epsilon}(x) (\D \zeta)\D x\arrowvert \nonumber \\
& + \sum_{i=1}^{\hat{m}}\arrowvert \int_{\sB_{N_1}} f(x)\int_{\Act} r_{i}(x)h_{j(i)}(\zeta)\bar{v}_{m}^{\epsilon}(x) (\D \zeta)\D x - \int_{\sB_{N_1}} \hat{f}_{k(i)}(x)\int_{\Act} h_{j(i)}(\zeta)\bar{v}_{m}^{\epsilon}(x) (\D \zeta)\D x\arrowvert\nonumber\\
& + \sum_{i=1}^{\hat{m}}\arrowvert \int_{\sB_{N_1}} f(x)\int_{\Act} r_{i}(x)p_i(\zeta) \bar{v}_{m}^{\epsilon}(x)(\D \zeta)\D x - \int_{\sB_{N_1}} f(x)\int_{\Act} r_{i}(x)h_{j(i)}(\zeta)\bar{v}_{m}^{\epsilon}(x) (\D \zeta)\D x\arrowvert\nonumber\\
& + \arrowvert \int_{\sB_{N_1}} f(x)\int_{\Act} g(x, \zeta)\bar{v}_{m}^{\epsilon}(x) (\D \zeta)\D x - \int_{\sB_{N_1}} f(x)\int_{\Act} \sum_{i}^{\hat{m}} r_{i}(x)p_i(\zeta) \bar{v}_{m}^{\epsilon}(x)(\D \zeta)\D x \arrowvert\nonumber\\
& \leq \frac{\epsilon}{4} + \sum_{l=1}^{N_1}\sum_{i=1}^{\hat{m}}\arrowvert  \int_{D_{l}} \hat{f}_{k(i)}(x)\int_{\Act} h_{j(i)}(\zeta)v(x) (\D \zeta)\D x - \int_{D_{l}} \hat{f}_{k(i)}(x)\int_{\Act} h_{j(i)}(\zeta)\bar{v}_{l,m}^{\epsilon_l}(x) (\D \zeta)\D x\arrowvert
\end{align} Now, choose $\epsilon_{l,i}$ for $l = 1,\dots , N_1$ and $i=1,\dots ,\hat{m}$ in such a way that $\sum_{l=1}^{N_1}\sum_{i=1}^{\hat{m}}\epsilon_{l,i} \leq \frac{\epsilon}{4}$\,. Thus, in view of \cref{EBT3} there exists $\hat{M}_{2} := \max\{M_{k(i),j(i)}^{l}: i=1,\dots , \hat{m}; l = 1,\dots , N_1\}$ (where $M_{k(i),j(i)}^{l}\in \NN$ is the constant obtained as in \cref{EBT3} for $i=1,\dots , \hat{m}; l = 1,\dots , N_1$). Therefore, from \cref{EBT7} and \cref{EBT10}, we conclude that
\begin{align}\label{EBT11}
\arrowvert \int_{\Rd}f(x)\int_{\Act} g(x,\zeta)v(x)(\D \zeta)\D x - \int_{\Rd}f(x)\int_{\Act} g(x,\zeta)\bar{v}_{m}^{\epsilon}(x)(\D \zeta)\D x \arrowvert 
\leq \epsilon\,,
\end{align} for all $m\geq \hat{M}_2$\,. This completes the proof of the theorem\,.
\end{itemize}
\end{proof}

\subsection{Denseness of Continuous Policies}
Following the discussions above, one can show that the space of continuous stationary policies are also dense in the space of stationary policies under Borkar topology. This is a useful result as continuity allows for many approximation results to be invoked with little effort (see e.g.  \cite[Assumption A2.3, pp. 322]{KD92} where convergence properties of invariant measures corresponding to time-discretizations are facilitated). 
\begin{theorem}\label{TDContP}
For each $v\in \Usm$ there exists a sequence of continuous policies $\{v_m\}_{m}$ in $\Usm$ such that
\begin{equation}\label{BorkarTopology2Cont}
\lim_{m\to\infty}\int_{\Rd}f(x)\int_{\Act}g(x,\zeta)v_{m}(x)(\D \zeta)\D x = \int_{\Rd}f(x)\int_{\Act}g(x,\zeta)v(x)(\D \zeta)\D x
\end{equation}
for all $f\in L^1(\Rd)\cap L^2(\Rd)$ and $g\in \cC_b(\Rd\times \Act)$ 
\end{theorem}
\begin{proof}
As earlier we have $\{f_i\}_{i\in\NN}$ is a countable dense set in $\Lp^{1}(\Rd)$\,. Now for each $i\in\NN$, define a finite measure
$\nu_i$ on $(\Rd, \sB(\Rd))$, given by
$$\nu_i(A) = \int_{A} |f_i(x)|\D x \quad \forall \,\,\, A\in \sB(\Rd)\,.$$
Let $v\in\Usm$. Then, as in the proof of Theorem~\ref{TDPCP}, by successive application of Lusin's theorem (see \cite[Theorem~7.5.2]{D02-book}) and Tietze's extension theorem (see \cite[Theorem~4.1]{DG51}), for any $\epsilon_i >0$ there exists a closed set $K_i\in \Rd$ and a continuous function $v^{i}: \Rd \to \pV$ such that $v^{i}\equiv v$ on $K_i$ and $\nu_i(\Rd\setminus K_i) < \epsilon_{i}$\,. Hence, for any $g\in\cC_b(\Rd\times \Act)$, we have 
\begin{align*}\label{EBT1Remark}
&\arrowvert \int_{\Rd}f_i(x) \int_{\Act} g(x,\zeta)v(x)(\D \zeta)\D x - \int_{\Rd}f_i(x)\int_{\Act} g(x,\zeta)v^i(x)(\D \zeta)\D x \arrowvert \nonumber\\
&\leq \arrowvert \int_{\Rd\setminus K_i} f_i(x) \int_{\Act}g(x,\zeta)v(x)(\D \zeta)\D x - \int_{\Rd\setminus K_i} f_i(x) \int_{\Act} g(x,\zeta)v^i(x)(\D \zeta)\D x \arrowvert \nonumber\\
&\leq 2\|g\|_{\infty}\int_{\Rd\setminus K_i}|f_i(x)|\D x \nonumber\\
& = 2\|g\|_{\infty}\nu_i(\Rd\setminus K_i) \leq 2\|g\|_{\infty}\epsilon_i\,.
\end{align*} Since $\{f_i\}_{i\in\NN}$ is dense in $\Lp^{1}(\Rd)$, by choosing $\epsilon_i$ appropriately, we obtain our result\,. 
\end{proof} 
\section{Near Optimality of Finite Models for Controlled Diffusions}\label{NOptiFinite}
First we prove the near optimality of quantized policies for the $\alpha$-discounted cost. 
\begin{theorem}\label{T1.2}
Suppose Assumptions (A1)-(A3) hold. Then for each $\epsilon >0$ there exists a policy $v_{\epsilon}^*\in \Usm$ with finite actions and piecewise constant policies $\bar{v}_{\epsilon}^* \in \Usm$ such that 
\begin{equation}\label{ET1.2A}
\cJ_{\alpha}^{v_{\epsilon}^*}(x, c) \leq \inf_{U\in \Uadm}\cJ_{\alpha}^{U}(x, c) + \epsilon \quad\text{and}\quad \cJ_{\alpha}^{\bar{v}_{\epsilon}^*}(x, c) \leq \inf_{U\in \Uadm}\cJ_{\alpha}^{U}(x, c) + \epsilon \quad\quad\text{for all}\,\, x\in\Rd\,.
\end{equation}
\end{theorem}
\begin{proof}
From \cite[Theorem~3.5.6]{ABG-book}, it follows that there exists $v^*\in \Usm$ such that $\cJ_{\alpha}^{v^*}(x, c) = \inf_{U\in \Uadm}\cJ_{\alpha}^{U}(x, c)$ for all $x\in \Rd$\,. Since the map $v\mapsto \cJ_{\alpha}^{v}(x, c)$ is continuous on $\Usm$ (see, Theorem~ \ref{T1.1}) and the space of quatized stationary policies are dense in $\Usm$ (see, Lemma~\ref{DenseBorkarTopo}), it follows that for each $\epsilon > 0$ there exists a quatized policy $v_{\epsilon}^*\in \Usm$ satisfying \cref{ET1.2A}\,. Similarly, since the peicewise constant policies are dense in $\Usm$ (see, Theorem~\ref{TDPCP}), we conclude that for any $\epsilon > 0$ there exists $\bar{v}_{\epsilon}^*\in\Usm$ which satisfies \cref{ET1.2A}\,. This completes the proof.
\end{proof}
We now show that for the cost upto an exit time, the quantized (finite action/ piecewise constant) policies are near optimal\,.
\begin{theorem}\label{T1.2ExitCost}
Suppose Assumptions (A1)-(A3) hold. Then for each $\epsilon >0$ there exists a policy $v_{\epsilon}^*\in \Usm$ with finite actions and piecewise constant policies $\bar{v}_{\epsilon}^* \in \Usm$ such that 
\begin{equation}\label{ET1.2ExitCostA}
\hat{\cJ}_{e}^{v_{\epsilon}^*}(x) \leq \inf_{U\in \Uadm}\hat{\cJ}_{e}^{U}(x) + \epsilon \quad\text{and}\quad \hat{\cJ}_{e}^{\bar{v}_{\epsilon}^*}(x) \leq \inf_{U\in \Uadm}\hat{\cJ}_{e}^{U}(x) + \epsilon \quad\quad\text{for all}\,\, x\in\Rd\,.
\end{equation}
\end{theorem}
\begin{proof}
From \cite[p. 229]{B05Survey}, we know that there exists $v^*\in \Usm$ such that $\hat{\cJ}_{e}^{v^*}(x) = \inf_{U\in \Uadm}\hat{\cJ}_{e}^{U}(x)$\,. Now form the  continuity of the map $v\to \hat{\cJ}_{e}^{v}(x)$ (see Theorem~\ref{T1.1Exit}) and the density results (see Section~\ref{DensePol}), it is easy to see that for any given $\epsilon > 0$ there exists policies $v_{\epsilon}^*\in \Usm$ with finite actions and piecewise constant policies $\bar{v}_{\epsilon}^* \in \Usm$ satisfying \cref{ET1.2ExitCostA}\,. This completes the proof of the theorem\,.   
\end{proof}
Next we prove the near optimality of the quantized policies for the ergodic cost under near-monotonicity assumption on the running cost\,. Let $$\Theta_{v} \df \{v_n\mid v_n \,\,\text{is the quantized policy defined as in \cref{DenseStra1} corresponding to}\,\, v \}$$ and
$$\bar{\Theta}_{v} \df \{\bar{v}_n\mid \bar{v}_n \,\,\text{is the quantized policy defined as in \cref{EBT4} corresponding to}\,\, v \}\,.$$
 In order to establish our result we are assuming that the invariant measures set $$\Gamma_{v^*} \df \{\eta_{v_n^*}\mid \eta_{v_n^*}\,\,\text{is the invariant measure corresponding to}\,\, v_n^*\in \Theta_{v^*}\}$$ and  
$$\bar{\Gamma}_{v^*} \df \{\eta_{\bar{v}_n^*}\mid \eta_{\bar{v}_n^*}\,\,\text{is the invariant measure corresponding to}\,\, \bar{v}_n^*\in \bar{\Theta}_{v^*}\}$$ are tight, where $v^*\in\Usm$ is an ergodic optimal control. The sufficient condition which assures the required tightness is the following: if there exists a non-negative inf-compact function $f\in\cC^2(\Rd)$ such that
$$\sL_{v_{n}^*} f(x) \leq \kappa_0 - f(x)\quad\text{and}\quad \sL_{\bar{v}_{n}^*} f(x) \leq \kappa_0 - f(x)$$ for some constant $\kappa_0 >0$\,.
\begin{theorem}\label{ErgodNearmOPT1}
Suppose that Assumptions (A1) - (A4) hold. Also, suppose that corresponding to the optimal policy $v^*\in \Usm$, the following set of invariant measures $\Gamma_{v^*}$ and  $\bar{\Gamma}_{v^*}$ are tight and the running cost $c$ is near monotone with respect to $\sup_{v_n^*\in\Theta_{v^*}}\sE_x(c, v_{n}^*)$ and $\sup_{\bar{v}_n^*\in\bar{\Theta}_{v^*}}\sE_x(c, \bar{v}_{n}^*)$, that is,
 $$\sup_{v_n^*\in\Theta_{v^*}}\sE_x(c, v_{n}^*) < \liminf_{\norm{x}\to\infty}\inf_{\zeta\in \Act} c(x,\zeta)\quad \text{and}\quad \sup_{\bar{v}_n^*\in\bar{\Theta}_{v^*}}\sE_x(c, \bar{v}_{n}^*) < \liminf_{\norm{x}\to\infty}\inf_{\zeta\in \Act} c(x,\zeta).$$
Then for any given $\epsilon >0$ there exists a policy $v_{\epsilon}\in \Usm$ with finite actions and a piecewise constant policy $\bar{v}_{\epsilon}\in \Usm$ such that
\begin{equation}\label{ENearmonOPTA1A}
\sE_x(c, v_{\epsilon}) \leq \sE^*(c) + \epsilon\quad \text{and}\quad \sE_x(c, \bar{v}_{\epsilon}) \leq \sE^*(c) + \epsilon\,. 
\end{equation}
\end{theorem}
\begin{proof}
From \cite[Theorem~3.6.10]{ABG-book}, we know there exits a stable $v^*\in \Usm$ such that $\sE_x(c, v^*) = \sE^*(c)$\,. Since, by our assumption, the set of invariant measures $\Gamma_{v^*}$ and  $\bar{\Gamma}_{v^*}$ are tight. Thus by the continuity result (see Theorem~\ref{ergodicnearmono1}) and the density results (see Lemma~\ref{DenseBorkarTopo}, Theorem~\ref{TDPCP}), we deduce that for each $\epsilon>0$ there exists $v_{\epsilon}\in \Usm$ with finite actions and piecewise constant policy $\bar{v}_{\epsilon}\in \Usm$ such that \cref{ENearmonOPTA1A} holds\,. This completes the proof.
\end{proof}
Now for the ergodic cost criterion, under the Lyapunov type stability assumption we prove near optimality of quantized policies. 
\begin{theorem}\label{TErgoOptApprox1}
Suppose that assumptions (A1) - (A3) and (A5) hold. Then for any given $\epsilon>0$ there exists a quantized policy $v_{\epsilon}\in \Usm$ with finite actions and a piecewise constant policy $\bar{v}_{\epsilon}^* \in \Usm$ such that
\begin{equation}\label{ENearmonOPTA1}
\sE_x(c, v_{\epsilon}) \leq \sE^*(c) + \epsilon \quad\text{and}\quad \sE_x(c, \bar{v}_{\epsilon}) \leq \sE^*(c) + \epsilon\,. 
\end{equation}
\end{theorem} 
\begin{proof}
From \cite[Theorem~3.7.14]{ABG-book}, we know that there exists $v^*\in \Usm$ such that $\sE_x(c, v^*) = \sE^*(c)$. Now, since the space of quantized polices and piecewise constant policies are dense in $\Usm$ (see, Lemma~\ref{DenseBorkarTopo} and Theorem~\ref{TDPCP}) and the map $v\to \inf_{\Rd}\sE_x(c, v)$ is continuous on $\Usm$ (see, Theorem~\ref{ergodicLyap1}). For any given $\epsilon>0$, one can find a quantized policy $v_{\epsilon}\in\Usm$ with finite actions and a piecewise constant policy $\bar{v}_{\epsilon}^* \in \Usm$ such that \cref{ENearmonOPTA1} holds.   
\end{proof}

\begin{remark}\label{RContStaNear1}
In view of the continuity (see Section~\ref{CDiscCost}, Section~\ref{CErgoCost})
and the denseness (see Theorem~\ref{TDContP}) results, we have the near optimality of continuous stationary policies\,.
\end{remark} 
\section{Finite Horizon Cost: Time Discretization of Markov Policies and Near Optimality of Piecewise Constant Policies}\label{TimeDMarkov}

Recall \cref{FiniteCost1} as our cost criterion for the finite horizon setup. We will present three results in this section, where the ultimate goal is to arrive at near optimality of piecewise constant policies. While this approximation problem is a well-studied problem \cite{KD92}, \cite{HK-02A}, \cite{RF-16A}, our proof method is rather direct and appears to be new. Under uniform Lipschitz continuity and uniform boundedness assumptions on the diffusion coefficients and running cost function, in \cite{KD92}, \cite{HK-02A}, \cite{RF-16A} the authors have established similar approximation results using numerical procedures\,.  

\subsection*{Continuity of Finite Horizon Cost on Markov Policies under the Borkar Topology}
For simplicity, in this subsection we are assuming that $a, b, c$ are uniformly bounded (it is possible to relax these boundedness assumptions). In particular we are assuming that
\begin{itemize}
\item[\hypertarget{B1}{{(B1)}}]
The functions $a, b, c$ are are uniformly bounded, i.e., 
\begin{equation*}
\sup_{(x,\zeta)\in \Rd\times \Act}\left[\abs{b(x,\zeta)} + \norm{a(x)} + \sum_{i}^{d} \norm{\frac{\partial{a}}{\partial x_i}(x)} + \abs{c(x, \zeta)}\right] \,\le\, \mathrm{K}\,.
\end{equation*} for some positive constant $\mathrm{K}$\,. Moreover, $H\in \Sob^{2,p,\mu}(\Rd)\cap \Lp^{\infty}(\Rd)$\,,\,\, $p\ge 2$\,. 
\end{itemize} In view of \cite[Theorem~3.3, p. 235]{BL84-book}, the optimality equation (or, the HJB equation)
\begin{align*}
&\frac{\partial \psi}{\partial t} + \inf_{\zeta\in \Act}\left[\sL_{\zeta}\psi + c(x, \zeta) \right] = 0 \\
& \psi(T,x) = H(x)
\end{align*} admits a unique solution $\psi\in \Sob^{1,2,p,\mu}((0, T)\times\Rd)\cap \Lp^{\infty}((0, T)\times\Rd)$\,,\,\, $p\ge 2$\,. Thus, by It\^{o}-Krylov formula (see the verification results as in \cite[Theorem~3.5.2]{HP09-book}), we know the existence of an optimal Markov policy, that is, there exists $v^*\in \Um$ such that $\cJ_{T}(x, v^*) = \cJ_{T}^*(x)$\,.

In the following theorem, we show that the finite horizon cost is continuous in $\Um$ with respect to the Borkar topology (see Definition~\ref{BKTP1})\,. 
\begin{theorem}\label{TContFHC}
Suppose Assumptions (A1), (A3) and (B1) hold. Then the map $v\mapsto \cJ_{T}(x, v)$ from $\Um$ to $\RR$ is continuous. 
\end{theorem}
\begin{proof}
Let $v_n$ be a sequence in $\Um$ such that $v_n \to v$ in $\Um$, for some $v\in\Um$\,. From \cite[Theorem~3.3, p. 235]{BL84-book}, we have that for each $n\in\NN$ there exists a unique solution $\psi_n\in\Sob^{1,2,p,\mu}((0, T)\times\Rd)\cap \Lp^{\infty}((0, T)\times\Rd)$\,,\,\, $p\ge 2$ to the following Poisson equation
\begin{align}\label{TContFHC1A}
&\frac{\partial \psi_n}{\partial t} + \left[\sL_{v_n}\psi_n + c(x, v_n(t,x)) \right] = 0 \nonumber\\
& \psi_n(T,x) = H(x)\,.
\end{align} By It\^{o}-Krylov formula, we deduce that
\begin{align}\label{TContFHC1B}
\psi_{n}(t,x) = \Exp_x^{v_n}\left[\int_t^{T} c(X_s, v_n(s, X_s)) \D{s} + H(X_T)\right]
\end{align} 
This gives us 
\begin{equation}\label{TContFHC1C}
\norm{\psi_n}_{\infty} \leq T\norm{c}_{\infty} + \norm{H}_{\infty}\,.
\end{equation} Rewriting \cref{TContFHC1A}, we get 
\begin{align*}
&\frac{\partial \psi_n}{\partial t} + \sL_{v_n}\psi_n + \lambda_0 \psi_n = \lambda_0 \psi_n - c(x, v_n(t,x))  \nonumber\\
& \psi_n(T,x) = H(x)\,,
\end{align*} for some fixed $\lambda_0 >0$\,. Thus, by parabolic pde estimate \cite[eq. (3.8), p. 234]{BL84-book}, we deduce that
\begin{equation}\label{TContFHC1D}
\norm{\psi_n}_{\Sob^{1,2,p,\mu}} \leq \kappa_1 \norm{\lambda_0 \psi_n - c(x, v_n(t,x))}_{\Lp^{p,\mu}}\,.
\end{equation} Hence, from \cref{TContFHC1C}, \cref{TContFHC1D}, it follows that
$\norm{\psi_n}_{\Sob^{1,2,p,\mu}} \leq \kappa_2$ for some positive constant $\kappa_2$ (independent of $n$)\,. Since $\Sob^{1,2,p,\mu}((0, T)\times\Rd)$ is a reflexive Banach space, as a corollary of Banach Alaoglu theorem, there exists $\psi^*\in\Sob^{1,2,p,\mu}((0, T)\times\Rd)$ such that along a subsequence (without loss of generality denoting by same sequence) 
\begin{equation}\label{TContFHC1E}
\begin{cases}
\psi_n \to & \psi^*\quad \text{in}\quad \Sob^{1,2,p,\mu}((0, T)\times\Rd)\quad\text{(weakly)}\\
\psi_n \to & \psi^*\quad \text{in}\quad \Sob^{0,1,p,\mu}((0, T)\times\Rd)\quad \text{(strongly)}\,.
\end{cases}       
\end{equation}  Since $v_n\to v$ in $\Um$, multiplying both sides of the \cref{TContFHC1A} by test function $\phi\in\cC_c^{\infty}((0, T)\times \Rd)$ and integrating, we get
\begin{align}\label{TContFHC1EA}
\int_{0}^{T}\int_{\Rd}\frac{\partial \psi_n}{\partial t}\phi(t,x)\D t \D x + & \int_{0}^{T}\int_{\Rd}\trace\bigl(a(x)\grad^2 \psi_n\bigr)\phi(t,x)\D t \D x \nonumber\\
& + \int_{0}^{T}\int_{\Rd}\{b(x,v_{n}(t,x))\cdot \grad \psi_n +  c(x, v_{n}(t,x))\}\phi(t,x)\D t \D x = 0\,.
\end{align}
In view of \cref{TContFHC1E}, letting $n\to \infty$, from \cref{TContFHC1EA} we obtain that 
\begin{align*}
\int_{0}^{T}\int_{\Rd}\frac{\partial \psi^*}{\partial t}\phi(t,x)\D t \D x + & \int_{0}^{T}\int_{\Rd}\trace\bigl(a(x)\grad^2 \psi^*\bigr)\phi(t,x)\D t \D x \nonumber\\
& + \int_{0}^{T}\int_{\Rd}\{b(x,v(t,x))\cdot \grad \psi^* +  c(x, v(t,x))\}\phi(t,x)\D t \D x = 0\,.
\end{align*}
This implies that $\psi^*\in\Sob^{1,2,p,\mu}((0, T)\times\Rd)$ satisfies
\begin{align}\label{TContFHC1F}
&\frac{\partial \psi^*}{\partial t} + \left[\sL_{v}\psi^* + c(x, v(t,x)) \right] = 0 \nonumber\\
& \psi(T,x) = H(x)\,.
\end{align} 
Again, by It\^{o}-Krylov formula, it follows that
\begin{align}\label{TContFHC1G}
\psi^{*}(t,x) = \Exp_x^{v}\left[\int_t^{T} c(X_s, v(s, X_s)) \D{s} + H(X_T)\right]\,.
\end{align} Therefore, from \cref{TContFHC1B} and \cref{TContFHC1G}, we conclude that $v\mapsto \cJ_{T}(x, v)$ from $\Um$ to $\RR$ is continuous.
\end{proof}

\subsection{Time Discretization of Markov Policies}
Following, and briefly modifying, our approach so far involving stationary policies, in this section we show that piece-wise constant Markov policies are dense in the space of Markov policies $\Um$\,. Also, using this result we deduce the near optimality of piece-wise constant Markov policies\,. 
\begin{theorem}\label{TDPCMP}
For any $v\in \Um$ there exists a sequence of piecewise constant policies $\{v_m\}_{m}$ such that
\begin{equation}\label{BorkarTopology3}
\lim_{m\to\infty}\int_{0}^{\infty}\int_{\Rd}f(x,t)\int_{\Act}g(x,t,\zeta)v_{m}(x,t)(\D \zeta)\D x \D t = \int_{0}^{\infty}\int_{\Rd}f(x,t)\int_{\Act}g(x,t,\zeta)v(x,t)(\D \zeta)\D x \D t
\end{equation}
for all $f\in L^1(\Rd\times [0, \infty))\cap L^2(\Rd\times [0, \infty))$ and $g\in \cC_b(\Rd\times [0, \infty)\times \Act)$ \,.
\end{theorem}
\begin{proof}
Let $\hat{\sB}_{0} = \emptyset$ and $\hat{\sB}_{n} = \sB_n\times [0, n)$. Then,  define $\hat{D}_{n} = \hat{\sB}_{n}\setminus \hat{\sB}_{n-1}$ for $n\in \NN$\,. Now, it is clear that $\Rd\times [0,\infty) = \cup_{n=1}^{\infty} \hat{D}_{n}$\,. Since $\bar{v}_{n}\,\df\, v\arrowvert_{\hat{D}_n} : \hat{D}_n \to \pV$ is a measurable map. As in Theorem~\ref{TDPCP}, by Lusin's theorem and Tietze's extension theorem, for any $\epsilon_n > 0$ there exists a compact set $\hat{K}_{n}^{\epsilon_n}\subset \hat{D}_n$ and a continuous function $\bar{v}_{n}^{\epsilon_n}: \hat{D}_n \to \pV$ such that $ \bar{v}_{n}^{\epsilon_n}\equiv \bar{v}_{n}$ on $\hat{K}_{n}^{\epsilon_n}$ and $\arrowvert(\hat{D}_n\setminus \hat{K}_{n}^{\epsilon_n}) \arrowvert < \epsilon_n$\,. 

Also, as in Theorem~\ref{TDPCP}, since $(\pV, d_{\mathscr{P}})$ is compact, for each $m\in \NN$ there exists a finite set $\widehat{\Lambda}_{m} = \{\mu_{m,1}, \mu_{m,2}, \dots , \mu_{m, k_m}\}$ and a quantizer $\widehat{Q}_{m}: \pV \to \widehat{\Lambda}_{m}$ which induces a partition $\{\widehat{U}_{m,i}\}_{i=1}^{k_m}$ of the space $\pV$\,. 

Now, for any $v\in\Um$ define $\hat{D}_{n,i}^m = (\bar{v}_{n}^{\epsilon_n})^{-1}(\widehat{U}_{m,i})$. It is easy to see that $\hat{D}_{n} = \cup_{i=1}^{k_m} \hat{D}_{n,i}^m$\,. Define
\begin{equation*}
\bar{v}_{n,m}^{\epsilon_n}(x) \df \sum_{i=1}^{k_m} \mu_{m,i}\Ind_{\{\hat{D}_{n,i}^m\}}(x)\quad\text{for all}\quad x\in \hat{D}_n\,\,\,\text{and}\,\,\,m\in \NN\,.
\end{equation*} Hence, as in the proof of Theorem~\ref{TDPCP} (see Step~$1$), for any $\hat{f}\in L^1(\Rd\times [0, \infty))\cap L^2(\Rd\times [0, \infty)), \hat{g}\in \cC_b(\Act)$, there exists a positive constant $\bar{M}_0$ (depending on $\hat{f}, \hat{g}$ and $\epsilon_n$) such that
\begin{align}\label{EBTM2}
\arrowvert \int_{\hat{D}_n}\hat{f}(x,t) \int_{\Act} \hat{g}(\zeta)\bar{v}_{n}(x,t)(\D \zeta)\D x \D t - \int_{\hat{D}_n}\hat{f}(x,t)\int_{\Act} \hat{g}(\zeta)\bar{v}_{n,m}^{\epsilon_n}(x,t)(\D \zeta)\D x \D t\arrowvert \leq \epsilon_n\,,
\end{align} for all $m\geq \bar{M}_0$\,.

Now, for any given $\epsilon > 0$, define 
\begin{equation}\label{EBTM3}
\Tilde{v}_{m}^{\epsilon} \df \sum_{n=1}^{\infty} \bar{v}_{n,m}^{\epsilon_n} \quad \text{for}\,\,\, m\in \NN\,. 
\end{equation} Since $\hat{f}\in\Lp^1(\Rd\times [0, \infty))$ there exists $N_0 \in \NN$ such that $\int_{\hat{\sB}_{N_0}^c}|\hat{f}(x,t)|\D x \D t < \frac{\epsilon}{4\|\hat{g}\|_{\infty}}$\,.
Then closely mimicking the argument of Theorem~\ref{TDPCP} (see Step~$2$), we have that for each $\epsilon >0$ there exists a positive constant $\hat{M}_0$ (depending on $\hat{f}, \hat{g}, \epsilon$) such that for all $m\geq \hat{M}_0$ 
\begin{equation}\label{EBTM4}
\arrowvert \int_{[0, \infty)}\int_{\Rd}\hat{f}(x,t)\int_{\Act} \hat{g}(\zeta)v(x,t)(\D \zeta)\D x \D t - \int_{[0, \infty)}\int_{\Rd}\hat{f}(x,t)\int_{\Act} \hat{g}(\zeta)\Tilde{v}_{m}^{\epsilon}(x,t)(\D \zeta)\D x \D t\arrowvert \leq \epsilon\,.
\end{equation}

Let $\{\hat{f}_k\}_{k\in\NN}$ and $\{h_{j}\}_{j\in\NN}$ be countable dense set in $\Lp^1(\Rd\times [0, \infty))$ and $\cC(\Act)$ respectively\,. Suppose that $f\in\Lp^{1}(\Rd\times [0, \infty))\cap \Lp^{2}(\Rd\times [0, \infty))$ and $g\in\cC_{b}(\Rd\times [0, \infty)\times \Act)$\,. Since $f\in \Lp^{1}(\Rd\times [0, \infty))$, for given $\epsilon > 0$ there exists $N_{1}\in \NN$ such that $\int_{\hat{\sB}_{N_1}^c} |f(x,t)|\D x \D t\leq \frac{\epsilon}{4\|g\|_{\infty}}$\,. We know that in $\cC_b(\bar{\hat{\sB}}_{N_1}\times \Act)$ the functions of the form   $\{\sum_{i}^{m} r_{i}(x,t)p_i(\zeta)\}_{m\in\NN}$ forms an algebra which contains constants, where $r_i\in \cC(\bar{\hat{\sB}}_{N_1})$ and $p_i \in \cC(\Act)$\,. Thus by Stone-Weierstrass theorem there exists $\hat{m}$ (large enough) such that 
\begin{equation}\label{EBTM5}
\sup_{\hat{\sB}_{N_1}\times \Act} |g(x,t,\zeta) - \sum_{i}^{\hat{m}} r_{i}(x,t)p_i(\zeta)| \leq \frac{\epsilon}{24\|f\|_{\Lp^1(\Rd\times [0, \infty))}}\,.
\end{equation}
Since $p_i \in \cC(\Act)$ one can choose $h_{j(i)} \in \cC(\Act)$ such that
\begin{equation}\label{EBTM6}
\sup_{\zeta\in \Act} |p_i(\zeta) - h_{j(i)}(\zeta)| \leq \frac{\epsilon}{24\|f\|_{\Lp^1(\Rd\times [0, \infty))}\|r_{i}\|_{\infty}}\,.
\end{equation}
Also, since $fr_i\in \Lp^1(\Rd\times [0, \infty))$ there exists $\hat{f}_{k(i)}$ such that
\begin{equation}\label{EBTM7}
\int_{\hat{\sB}_{N_1}} |f(x,t)r_i(x,t) - \hat{f}_{k(i)}(x,t)|\D x \D t \leq \frac{\epsilon}{24\|f\|_{\Lp^1(\Rd\times [0, \infty))}\|h_{i}\|_{\infty}}\,.
\end{equation} Thus, in view of \cref{EBTM5}, \cref{EBTM6}, \cref{EBTM7}, following the steps of Theorem~\ref{TDPCP} (see Step~$3$) we conclude that
\begin{align*}
\arrowvert \int_{0}^{\infty}\int_{\Rd}f(x,t)\int_{\Act} g(x,t,\zeta)v(x,t)(\D \zeta)\D x \D t - \int_{0}^{\infty}\int_{\Rd}f(x,t)\int_{\Act} g(x,t,\zeta)\bar{v}_{m}^{\epsilon}(x,t)(\D \zeta)\D x \D t\arrowvert 
\leq \epsilon\,,
\end{align*} for all $m\geq \hat{M}_1$, for some positive constant $\hat{M}_1$\,. This completes the proof of the theorem\,.
\end{proof}
\subsection*{Near Optimality of Piecewise Constant Policies for Finite Horizon Cost}

Now, from Theorem~\ref{TDPCMP} and Theorem~\ref{TContFHC}, we have the following near-optimality results\,.
\begin{theorem}\label{TFiniteOptApprox1}
Suppose that assumptions (A1),(A3) and (B1) hold. Then for any given $\epsilon>0$ there exists a piecewise constant policy $\bar{v}_{\epsilon}^* \in \Um$ such that
\begin{equation}\label{TFiniteOptApprox1A}
\cJ_{T}(x, \bar{v}_{\epsilon}^*) \leq \cJ_{T}^* + \epsilon \quad\text{for all} \quad x\in\Rd\,. 
\end{equation}
\end{theorem} 
\begin{proof}
From our previous discussion, we know that there exists $v^*\in \Um$ such that $\cJ_{T}(x, v^*) = \cJ_{T}^*$\,. Since the space of piecewise constant policies are dense in $\Um$ (see Theorem~\ref{TDPCMP}) and the map $v\mapsto \cJ_{T}(x, v)$ is continuous on $\Um$ (see Theorem~\ref{TContFHC}), for any given $\epsilon>0$, one can find a piecewise constant policy $\bar{v}_{\epsilon}^* \in \Um$ such that \cref{TFiniteOptApprox1A} holds\,.   
\end{proof}
\begin{remark}
In view of the existence results as in \cite[Chapter~4]{LSU67-book}, in obtaining the near optimality of piecewise constant Markov policies for finite horizon costs, one can relax the uniform boundedness assumption (B1), in particular, under (A1)-(A3) we can deduce similar results\,. Which extends the results of \cite{KD92}, \cite{HK-02A}, \cite{RF-16A} to a more general control model\,.
\end{remark}
\section*{Conclusion}
We studied regularity properties of induced cost (under several criteria) on a controlled diffusion process with respect to a control policy space defined by Borkar \cite{Bor89}. We then studied implications of these properties on existence and, in particular, approximations for optimal controlled diffusions. Via such a unified approach, we arrived at very general approximation results for optimal control policies by quantized (finite action / piecewise constant) stationary control policies for a general class of controlled diffusions in the whole space $\Rd$\, as well as time-discretizations for the criteria with finite horizons. 

\bibliography{Quantization}

\end{document}